\documentclass{article}

\usepackage[a4paper]{geometry}

\usepackage{amsfonts}

\usepackage{stmaryrd}

\usepackage{amsthm}

\usepackage{amssymb}

\usepackage[leqno]{amsmath}

\usepackage[all]{xy}

\usepackage{longtable}

\usepackage{mathptmx}

\def\N{\mathbb{N}}
\def\Z{\mathbb{Z}}
\def\C{\mathbb{C}}
\def\R{\mathbb{R}}
\def\K{\mathcal{K}}
\def\B{\mathcal{B}}

\def\xx{\sim}

\DeclareMathOperator{\diag}{diag}

\DeclareMathOperator{\Span}{span}

\DeclareMathOperator{\Cliff}{Cliff}

\title{Maximal and reduced Roe algebras of coarsely embeddable spaces}

\author{J\'{a}n \v{S}pakula\footnote{The first author was supported by the \emph{Deutsche
Forschungsgemeinschaft} (SFB 478 and SFB 878).} and Rufus Willett}

\theoremstyle{plain} \newtheorem{amen}{Theorem}[section]

\theoremstyle{plain} \newtheorem{atmen}[amen]{Theorem}

\theoremstyle{plain} \newtheorem{aprop}[amen]{Proposition}

\theoremstyle{plain} \newtheorem{main}[amen]{Theorem}

\theoremstyle{remark} \newtheorem{groupoids}[amen]{Remark}

\theoremstyle{remark} \newtheorem{tex}[amen]{Example}

\theoremstyle{remark} \newtheorem{cex}[amen]{Example}

\theoremstyle{remark} \newtheorem{freeex}[amen]{Example}

\theoremstyle{remark} \newtheorem{agsex}[amen]{Example}

\theoremstyle{remark} \newtheorem{serreex}[amen]{Example}

\theoremstyle{plain} \newtheorem{serthe}[amen]{Theorem}

\theoremstyle{definition} \newtheorem{bg}[amen]{Definition}

\theoremstyle{definition} \newtheorem{matrixnotation}[amen]{Notation}

\theoremstyle{definition} \newtheorem{algroe}[amen]{Definition}

\theoremstyle{definition} \newtheorem{roe}[amen]{Definition}

\theoremstyle{plain} \newtheorem{fundlem}[amen]{Lemma}

\theoremstyle{definition} \newtheorem{maxroe}[amen]{Definition}

\theoremstyle{definition} \newtheorem{reducednotation}[amen]{Notation}

\theoremstyle{definition} \newtheorem{propa}{Definition}[section]

\theoremstyle{plain} \newtheorem{finlem}[propa]{Lemma}

\theoremstyle{plain} \newtheorem{extendlem}[propa]{Corollary}

\theoremstyle{plain} \newtheorem{closlem}[propa]{Lemma}

\theoremstyle{definition} \newtheorem{cembed}{Definition}[section]

\theoremstyle{definition} \newtheorem{adef}[cembed]{Definition}

\theoremstyle{definition} \newtheorem{wdef}[cembed]{Definition}

\theoremstyle{definition} \newtheorem{supp0}[cembed]{Definition}

\theoremstyle{definition} \newtheorem{tra}[cembed]{Definition}

\theoremstyle{plain} \newtheorem{mer}[cembed]{Proposition}

\theoremstyle{definition} \newtheorem{nuclearity}[cembed]{Remark}

\theoremstyle{definition} \newtheorem{suppo}[cembed]{Definition}

\theoremstyle{plain} \newtheorem{invclo}[cembed]{Lemma}

\theoremstyle{plain} \newtheorem{disj}[cembed]{Lemma}

\theoremstyle{plain} \newtheorem{geomlem}[cembed]{Lemma}

\theoremstyle{plain} \newtheorem{pastelem}[cembed]{Lemma}

\theoremstyle{definition} \newtheorem{ucalgdef}{Definition}[section]

\theoremstyle{plain} \newtheorem{gradelem}[ucalgdef]{Lemma}

\theoremstyle{remark} \newtheorem{ucalgrem}[ucalgdef]{Remark}

\theoremstyle{plain} \newtheorem{ucalglem}[ucalgdef]{Lemma}

\theoremstyle{definition} \newtheorem{ashomo}[ucalgdef]{Claim}

\theoremstyle{definition} \newtheorem{completions}[ucalgdef]{Claim}

\theoremstyle{plain} \newtheorem{composition}[ucalgdef]{Theorem}

\theoremstyle{plain} \newtheorem{morita}[ucalgdef]{Proposition}

\theoremstyle{plain}\newtheorem{kmap}[ucalgdef]{Proposition}

\theoremstyle{plain} \newtheorem{wrapup}[ucalgdef]{Corollary}

\theoremstyle{definition} \newtheorem{asnot}{Definition}[section]

\theoremstyle{plain} \newtheorem{as-compose}[asnot]{Lemma}

\theoremstyle{remark} \newtheorem{acrem}[asnot]{Remark}

\begin{document}

\maketitle

\abstract{In \cite{Gong:2008ja}, Gong, Wang and Yu introduced a maximal, or universal, version of the Roe $C^*$-algebra associated to a metric space.  We study the relationship between this maximal Roe algebra and the usual version, in both the uniform and non-uniform cases.  The main result is that if a (uniformly discrete, bounded geometry) metric space $X$ coarsely embeds in a Hilbert space, then the canonical map between the maximal and usual (uniform) Roe algebras induces an isomorphism on $K$-theory.  We also give a simple proof that if $X$ has property A, then the maximal and usual (uniform) Roe algebras are the same.  These two results are natural coarse-geometric analogues of certain well-known implications of a-T-menability and amenability for group $C^*$-algebras.  The techniques used are $E$-theoretic, building on work of Higson-Kasparov-Trout \cite{Higson:1999be}, \cite{Higson:2001eb} and Yu \cite{Yu:200ve}. \\

}

\section{Introduction}

Say $G$ is a second countable, locally compact group, and $C_{max}^*(G)$, $C^*_\lambda(G)$ are respectively its maximal and reduced group $C^*$-algebras.  One then has the following theorem of Hulanicki \cite{Hulanicki:1967lt}.

\begin{amen}\label{amen}
$G$ is amenable if and only if the canonical quotient map $\lambda:C_{max}^*(G)\to C_\lambda^*(G)$ is an isomorphism. \qed
\end{amen}

Much more recently, as part of their deep work on the Baum-Connes conjecture \cite{Higson:2001eb}, N. Higson and G. Kasparov proved the following theorem.

\begin{atmen}\label{atmen}
If $G$ is a-T-menable then the canonical quotient map $\lambda:C_{max}^*(G)\to C_\lambda^*(G)$ induces an isomorphism on $K$-theory. \qed
\end{atmen}

The aim of this work is to prove partial analogues of these theorems in the setting of coarse geometry.  The part of amenability is here played by G. Yu's \emph{property A}, and that of a-T-menability by \emph{coarse embeddability in Hilbert space}  (definitions are given in the main body of the paper).   The part of the group $C^*$-algebras is played by the \emph{Roe algebras}, or alternatively the \emph{uniform Roe algebras}, of a metric space $X$.  

The Roe algebra of $X$, denoted $C^*(X)$, and uniform Roe algebra, denoted $C^*_u(X)$, were introduced by J. Roe in his work on index theory on open manifolds \cite{Roe:1988qy}, \cite{Roe:1993lq} and have played a fundamental r\^{o}le in the subsequent development of $C^*$-algebraic approaches to large-scale index theory and coarse geometry.  More recently, a maximal version of the Roe algebra, $C^*_{max}(X)$, was introduced in work of Gong, Wang and Yu on the Baum-Connes and coarse Baum-Connes conjectures \cite{Gong:2008ja}; a maximal version of the uniform Roe algebra, $C^*_{u,max}(X)$, can de defined analogously.  

Just as in the group case, then, there are canonical quotient maps
\begin{displaymath}
\lambda:C^*_{max}(X)\to C^*(X) \textrm{ and } \lambda:C^*_{u,max}(X)\to C^*_u(X).
\end{displaymath}

Our main results in this paper are the following analogues of Theorems \ref{amen} and \ref{atmen}.

\begin{aprop}\label{aprop}
Let $X$ be a bounded geometry, uniformly discrete metric space with property A.  Then the canonical quotients $\lambda:C^*_{u,max}(X)\to C^*_u(X)$ and $\lambda:C^*_{max}(X)\to{}C^*(X)$ are isomorphisms.
\end{aprop}

\begin{main}\label{main}
Let $X$ be a bounded geometry, uniformly discrete metric space that coarsely embeds in Hilbert space.  Then the canonical quotients $\lambda:C^*_{u,max}(X)\to{}C^*_u(X)$ and $\lambda:C^*_{max}(X)\to{}C^*(X)$ induce isomorphisms on $K$-theory.
\end{main}

Proposition \ref{aprop} is straightforward: in fact it is a special case of \cite[Corollary 5.6.17]{Brown:2008qy}, but it fits well into the philosophy of this piece, and we give a simple, direct proof below.  The proof we give of Theorem \ref{main} is substantially more involved, relying heavily on deep work of Yu \cite{Yu:200ve} and Higson-Kasparov-Trout \cite{Higson:1999be},\cite{Higson:2001eb}.  

Here is an outline of the paper.  Section \ref{defsec} introduces definitions and notation.  Section \ref{exsec} gives some examples of the sort of `wild' behaviour that can occur when the hypotheses of Proposition \ref{aprop} and Theorem \ref{main} fail, as well as asking some questions.  We give our proof of Proposition \ref{aprop} in Section \ref{asec}.  Unfortunately, we do not know whether the converse is true (it is in fact a special case of a general open problem for groupoid $C^*$-algebras - see for example Remark 6.1.9 in \cite{Anantharaman-Delaroche:2000mw}).  Section \ref{twist} introduces a variant of the twisted Roe algebra of Yu \cite[Section 5]{Yu:200ve}, and proves that the maximal and reduced versions of this algebra are always isomorphic (this is an analogue of the fact that the maximal and reduced crossed product $C^*$-algebras associated to a \emph{proper} action are isomorphic).  Section \ref{mainproof} then puts all this together with variants of Yu's coarse Dirac and Bott asymptotic morphisms \cite[Section 7]{Yu:200ve} to prove Theorem \ref{main}. The main idea of the proof is to use the coarse Dirac and Bott morphisms to `replace' the $K$-theories of  $C^*_{u}(X)$ and $C^*_{u,max}(X)$ by the $K$-theories of their twisted versions; but those are isomorphic even on the $C^*$-level by Section \ref{twist}. One ingredient of the proof which seems of independent interest is a (strong) Morita equivalence between the (maximal) uniform algebra of a space $X$ and the associated (maximal) uniform Roe algebra: both the uniform algebra, and uniform Roe algebra, of $X$ are completions of a certain algebra of $X$-by-$X$ matrices $(T(x,y))_{x,y\in X}$; the difference between the two is that for the uniform algebra, the entries $T(x,y)$ are allowed to be compact operators of uniformly finite rank, while for the uniform Roe algebra, the entries are complex scalars (see Defintions \ref{algroe}, \ref{roe} and \ref{ucalgdef} below).

\begin{groupoids}\label{groupoids}
Let us for a moment assume that the groupoid-equivariant $KK$-theory machinery works for non-$\sigma$-unital, non-separable $C^*$-algebras and groupoids with non-second countable base space. Then one can derive the main results of this paper by using known results as follows.  To every metric (or more generally, coarse) space $X$, one can associate the so-called coarse groupoid $\mathcal{G}(X)$ \cite{Skandalis:2002ng}, \cite[Chapter 10]{Roe:2003rw}. Furthermore, the reduced and maximal groupoid $C^*$-algebras of $\mathcal{G}(X)$ are precisely the uniform Roe algebra and the maximal uniform Roe algebra of $X$. Next, it is shown in \cite{Skandalis:2002ng} that property A of $X$ corresponds to (topological) amenability of $\mathcal{G}(X)$ and in \cite{Tu:1999qm} that coarse embeddability of $X$ into a Hilbert space corresponds to the Haagerup property for $\mathcal{G}(X)$.  Consequently, Theorem \ref{main} follows from the deep result of Tu \cite{Tu:1999bq} (generalising the work of Higson and Kasparov \cite{Higson:2001eb}) that groupoids with the Haagerup property are $K$-amenable, whence the quotient map from their maximal groupoid $C^*$-algebra to their reduced $C^*$-algebra is a $KK$-equivalence.

However, there are serious technical problems with the outline above, vis: $KK$-products  do not necessarily exist if the algebras involved are not $\sigma$-unital; even if the products exist, they are a priori not unique; if products $a\circ(b\circ c)$ and $(a\circ b)\circ c$ exist for $KK$-elements defined using non-$\sigma$-unital algebras, they are not a priori equal.  We thus thought a direct proof of Theorem \ref{main} was merited, as although most experts would expect the Theorem to be true for by analogy with the separable case, it does not appear to follow from anything in the literature.  Indeed, we hope our work is of some interest insofar as it suggests methods that can be employed when dealing with non-separable algebras in this sort of context.

At the moment, it remains open whether our proof can be improved to obtain that the canonical quotient map $\lambda$ from Theorem \ref{main} induces an equivalence in $E$-theory. There are two missing pieces: computing the composition $\beta_*\circ\alpha_*$ (see Theorem \ref{composition}), and perhaps showing that strong Morita equivalence implies $E$-equivalence in full generality (all proofs in the literature use at least $\sigma$-unitality). One consequence of this would be `$E$-nuclearity' of the uniform Roe algebra of a coarsely embeddable space (see Remark \ref{nuclearity}).
\end{groupoids}

\subsection*{Acknowledgments}
This work was started during the focused semester on $KK$-theory and its applications, held at the West\-f\"{a}\-li\-sche Wilhelms-Universit\"{a}t, M\"{u}nster, in summer 2009.  The authors would like to thank the organisers (notably Siegfried Echterhoff, Walther Paravicini and Pierre Clare) for a very inspiring and valuable semester. The authors would further like to thank Guoliang Yu for encouragement to pursue this research and Erik Guentner for helpful discussions about $E$-theory.

\subsection{Definitions and notation}\label{defsec}

We work in the context of metric spaces, rather than general coarse structures.  The reader can easily check that the proof of Proposition \ref{aprop} goes through for a general bounded geometry coarse space; however, the hypotheses for our main result Theorem \ref{main} force a coarse space satisfying them to be metrisable.  Thus we are not really losing any generality by restricting to the case of metric spaces.

\begin{bg}
Let $(X,d)$ be a metric space, and for any $r>0$, $x\in X$, let $B_r(x):=\{y\in X~|~d(x,y)<r\}$ denote the open ball of radius $r$ about $x$.  $X$ is said to be \emph{uniformly discrete} if there exists $\delta>0$ such that for all $x,y\in X$, if $x\neq{}y$ then $d(x,y)\geq\delta$.  $X$ is said to be of \emph{bounded geometry} if for all $r>0$ there exists $N_r\in\N$ such that for all $x\in X$, $|B_r(x)|\leq{}N_r$.  An \emph{entourage} for $X$ is a set of the form
\begin{equation}\label{entourage}
\{(x,y)\in X\times X~|~d(x,y)\leq S\}
\end{equation} 
for some $S\in \R_+$, or a subset of such a set.
\end{bg}

\emph{For the remainder of this note, $X$ will denote a uniformly discrete, bounded geometry metric space}.  Interesting examples include finitely generated discrete groups equipped with word metrics, uniformly discrete subsets of Riemannian manifolds that have bounded curvature and positive injectivity radius, and the box spaces from Section \ref{exsec}.

\begin{matrixnotation}\label{matrixnotation}
Throughout this note, we will deal with $X$-by-$X$ indexed matrices (with entries in one of several $*$-algebras).  If $T$ is such a matrix, for consistency with the notation from \cite{Yu:200ve}, we write $T(x,y)$ for the $(x,y)^\textrm{th}$ entry.
\end{matrixnotation}

\begin{algroe}\label{algroe}
Let $T=(T(x,y))$ be an $X$-by-$X$ matrix, where each $T(x,y)$ is an entry in some algebra.  $T$ is said to be of \emph{finite propagation} if there exists $S>0$ such that $T(x,y)=0$ whenever $d(x,y)\geq S$ (i.e.\ the only non-zero matrix coefficients of $T$ occur in an entourage as in line (\ref{entourage}) above). 

The \emph{algebraic uniform Roe algebra of $X$}, denoted $\C_u[X]$, is the set of $X$-by-$X$ complex matrices of finite propagation with uniformly bounded entries.  It is a $*$-algebra when equipped with the usual matrix operations (note that multiplication makes sense, as only finitely many elements in each `row' and `column' can be non-zero).  For fixed $S>0$, we denote by $\C^S_u[X]$ the subspace consisting of those $T$ such that $T(x,y)=0$ for all $x,y$ with $d(x,y)>S$.

Fix now a separable infinite dimensional Hilbert space $\mathcal{H}$ and let $\mathcal{K}:=\mathcal{K}(\mathcal{H})$ denote the compact operators on $\mathcal{H}$.  The \emph{algebraic Roe algebra of $X$}, denoted $\C[X,\mathcal{K}]$, is the set of $X$-by-$X$ matrices $T$ of finite propagation and with uniformly bounded entries from $\mathcal{K}(\mathcal{H})$.  $\C[X,\mathcal{K}]$ is equipped with a $*$-algebra structure using the usual matrix operations and the $*$-algebra structure on $\mathcal{K}(\mathcal{H})$.  $\C^S[X,\mathcal{K}]$ is defined analogously to the uniform case.
\end{algroe}

Note that $\C_u[X]$ admits a natural $*$-representation by `matrix multiplication' on $\ell^2(X)$, and similarly $\C[X,\mathcal{K}]$ admits a natural $*$-representation on $\ell^2(X,\mathcal{H})$.

\begin{roe}\label{roe}
The \emph{uniform Roe algebra of $X$}, denoted $C^*_u(X)$, is the completion of $\C_u[X]$ for its natural representation on $\ell^2(X)$.  The \emph{Roe algebra of $X$}, denoted $C^*(X)$, is the completion of $\C[X,\mathcal{K}]$ for its natural representation on $\ell^2(X,\mathcal{H})$.
\end{roe}

The following fundamental lemma is essentially proved in \cite[Section 3]{Gong:2008ja}; see also \cite[Lemma 4.27]{Roe:2003rw} for a similar idea.  It will be used several times below, and is moreover needed to show that the maximal (uniform) Roe algebra is well-defined.

\begin{fundlem}\label{fundlem}
For all $S>0$ there exists a constant $C_S$ such that for all $T\in\C_u^S[X]$ (respectively, $T\in\C^S[X,\mathcal{K}]$) and any $*$-representation $\pi:\C_u^S[X]\to\mathcal{B}(\mathcal{H})$ (resp. $\pi:\C_u^S[X,\mathcal{K}]\to\mathcal{B}(\mathcal{H})$) one has that 
\begin{displaymath}
\|\pi(T)\|_{\mathcal{B}(\mathcal{H})}\leq C_S\sup_{x,y\in{}X}\|T(x,y)\|. 
\end{displaymath}
\end{fundlem}

\begin{proof}
As we will need slight variants of this lemma several times, for the reader's benefit we sketch a proof.  The essential point is that for any $S>0$ there exist $C_S$ partial isometries $v_1,...,v_{C_S}$ in $\C^S_u[X]$ such that any $T\in \C_u^S[X]$ can be written (uniquely) as 
\begin{displaymath}
T=\sum_{i=1}^{C_S}f_iv_i,
\end{displaymath}
where each $f_i$ is an element of $l^\infty(X)$.  The proof for $\C^S[X,\mathcal{K}]$ is similar, but one replaces `$l^\infty(X)$' with `$l^\infty(X,\mathcal{K})$', and notes that partial isometries as $v_1,...,v_{C_S}$ as above are only multipliers of $\C[X,\mathcal{K}]$. 
\end{proof}

\begin{maxroe}\label{maxroe}
The \emph{maximal uniform Roe algebra of $X$}, denoted $C^*_{u,max}(X)$, is the completion of $\C_u[X]$ for the norm
\begin{displaymath}
\|T\|=\sup\{\|\pi(T)\|_{\mathcal{B}(\mathcal{H})}~|~\pi:\C_u[X]\to\mathcal{B}(\mathcal{H}) \textrm{ a $*$-representation}\}.
\end{displaymath}
One defines the \emph{maximal Roe algebra}, denoted $C^*_{max}(X)$, analogously.
\end{maxroe}

\begin{reducednotation}\label{reducednotation}
We will sometimes call $C^*(X)$ ($C^*_u(X)$) the \emph{reduced} (uniform) Roe algebra when we want to emphasise that we are not talking about the maximal case.

We will also often write `$C^*_{u,\xx}(X)$' in a statement if it applies to both $C^*_u(X)$ and $C^*_{u,max}(X)$ to avoid repeating it, and similarly for the other variants of the Roe algebra that we introduce throughout the piece.  This employs the convention that `$\xx$' means the same thing when appearing twice in a clause; for example,
\begin{displaymath}
\phi:C^*_{\xx}(X)\to C^*_{u,\xx}(X)
\end{displaymath}
means that $\phi$ can refer to either a map from $C^*(X)$ to $C^*_{u}(X)$, or to a map from $C^*_{max}(X)$ to $C^*_{u,max}(X)$, but not to a map that mixes the maximal and reduced versions.
\end{reducednotation}

From Section \ref{asec} onwards, we will work exclusively with the uniform algebras $C^*_{u,\xx}(X)$ (and some variants that we need for proofs).  The proof of Proposition \ref{aprop} for the non-uniform algebras $C^*_\xx(X)$ is precisely analogous.  The proof of Theorem \ref{main} for $C^*_\xx(X)$ is similar to, but significantly simpler than, that for the uniform case $C^*_{u,\xx}(X)$, and also closer to the material in \cite{Yu:200ve}.

\subsection{Examples and Questions}\label{exsec}

We conclude the introduction with some examples that help motivate the main results.  All of the examples we give are so-called \emph{Box spaces} associated to a discrete group $\Gamma$; we sketch the definition in the next paragraph.

Recall first that if $\Gamma$ is a finitely generated discrete group, and $(\Gamma_k)_{k\in\N}$ a nested (i.e.\ $\Gamma_{k+1}\leq \Gamma_k$) sequence of finite index subgroups of $\Gamma$ such that $\cap_k\Gamma_k=\{e\}$, then one can build an associated metric space out of the disjoint union of the (finite) spaces $\Gamma/\Gamma_k$ called the \emph{box space} of the pair $(\Gamma,(\Gamma_k))$ and denoted $X(\Gamma,(\Gamma_k))$.  See for example \cite[start of Section 2.2]{Oyono-Oyono:2009ua}.

Examples \ref{tex} and \ref{cex} discuss cases where the conclusions (and the hypotheses!) of Proposition \ref{aprop} and Theorem \ref{main} fail, in both the uniform and non-uniform cases; the existence of such is not obvious.  Examples \ref{cex}, \ref{agsex} discuss interesting borderline cases that we do not currently know how to deal with.  Example \ref{serreex} discusses an intriguing connection with, and potential application to, a long-standing open problem in number theory.

\begin{tex}\label{tex}
Say $\Gamma=SL(2,\Z)$, and that for each $k\in\N$, $\Gamma_k$ is the kernel of the natural map $SL(2,\Z)\to SL(2,\Z/2^k\Z)$.  Then $X:=X(\Gamma,(\Gamma_k))$ is an expander \cite[Example 4.3.3, D]{Lubotzky:1994tw}; as is well known (e.g. \ \cite[Proposition 11.26]{Roe:2003rw}), $X$ does not coarsely embed into Hilbert space (and thus also does not have property A).  Now, there is a natural inclusion $\C[\Gamma]\to \C_u[X]$ of the group algebra into the algebraic uniform Roe algebra, whence a commutative diagram
\begin{equation}\label{gin}
\xymatrix{ C^*_{max}(\Gamma) \ar[d]^= \ar[r] & C^*_{u,max}(X) \ar[d]^\lambda \\ C^*_{max}(\Gamma) \ar[r] & C^*_u(X) }
\end{equation}
(the left-hand-sides are the maximal group $C^*$-algebra, \emph{not} the maximal Roe algebra of $\Gamma$, which we would denote $C_{max}^*(|\Gamma|)$).
It is not hard to check that the top horizontal map is an injection.  Note, however, that $SL(2,\Z)$ has property $(\tau)$ with respect to the family of congruence subgroups, but not property (T), whence representations factoring through congruence subgroups are not dense in the unitary dual $\hat{\Gamma}$; hence the bottom horizontal map is not an injection.  We must therefore have that $C^*_{u,max}(X)\neq C^*_u(X)$; essentially the same argument applies to show that $C^*_{max}(X)\neq C^*(X)$.
\end{tex}

\begin{cex}\label{cex}
Let $\Gamma$, $(\Gamma_k)$, $X$ be as in the previous example.  In \cite[Example 4.20]{Oyono-Oyono:2009ua}, H. Oyono-Oyono and G. Yu point out that for this space, the maximal coarse assembly map $\mu_{X,max,*}$ is an isomorphism, while the (`usual') coarse assembly map $\mu_{X,*}$ is not surjective.  As, however, there is a commutative diagram
\begin{displaymath}
\xymatrix{ \lim_rK_*(P_r(X)) \ar[rr]^{\mu_{X,max,*}} \ar[drr]^{\mu_{X,*}} & & K_*(C^*_{max}(X)) \ar[d]^{\lambda_*} \\ & & K_*(C^*(X))},
\end{displaymath}
$\lambda_*$ cannot be surjective either.  We expect that a similar phenomenon occurs in the uniform case, but currently we have no proof of this fact.
\end{cex}

\begin{freeex}\label{freeex}
Say again that $\Gamma=SL(2,\Z)$, and for each $k\in\N$, let $\Gamma_k$ be the intersection over the kernels of all homomorphisms from $\Gamma$ to a group of cardinality at most $k$, so each $\Gamma_k$ is a finite index normal subgroup of $\Gamma$.  Using that $SL(2,\Z)$ has property (FD) \cite[Section 2]{Lubotzky:2004xw}, but not property (T), the sequence of quotients forming $X:=X(\Gamma,(\Gamma_k))$ is \emph{not} an expander.  As $\Gamma$ is not amenable, $X$ also does not have property A \cite[Proposition 11.39]{Roe:2003rw}.  We do not know the answers to the following questions:
\begin{enumerate}
\item Does $X$ coarsely embed in Hilbert space?
\item Is the map $\lambda$ from the maximal (uniform) Roe algebra of $X$ to the (uniform) Roe algebra of $X$ an isomorphism?
\item Is the coarse Baum-Connes conjecture true for $X$?
\item Does the map $\lambda$ from the maximal Roe algebra of $X$ to the Roe algebra of $X$ induce an isomorphism on $K$-theory? 
\end{enumerate}
Note that in the diagram in line (\ref{gin}) above, both horizontal maps are injections for this example, so this obstruction to question (2) no longer exists.  A positive answer to (1) or (2) implies positive answers to (3) and (4); moreover, (3) and (4) are equivalent by results of H. Oyono-Oyono and G. Yu \cite{Oyono-Oyono:2009ua}.  Any answers at all (positive or negative) would provide interesting new examples in coarse geometry, and some would be of broader interest.  
\end{freeex}

\begin{agsex}\label{agsex}
In recent work of G. Arzhantseva, E. Guentner and the first author, an example of a space $X(\Gamma,(\Gamma_k))$ which does not have property A, yet does coarsely embed in Hilbert space is constructed (here $\Gamma$ is a free group).  It would be interesting to know the answer to question (2) above (the other answers are all `yes') for this space.
\end{agsex}

\begin{serreex}\label{serreex}
Recall a conjecture of Serre \cite{Serre:1970df} that arithmetic lattices in $SO(n,1)$ do not have the congruence subgroup property.

\begin{serthe}\label{serthe}
Let $\Gamma$ be an arithmetic lattice in $SO(n,1)$.  If there exists a nested family $(\Gamma_k)$ of normal subgroups of $\Gamma$ with trivial intersection, and if the answer to any of the questions (1) through (4) from Example \ref{freeex} is `yes' for the space $X:=X(\Gamma,(\Gamma_k))$, then $\Gamma$ satisfies Serre's conjecture.
\end{serthe} 

\begin{proof}
A result of N. Higson \cite{Higson:1999km} shows that if $\Gamma$ has property $(\tau)$ with respect to $(\Gamma_k)$, then the answer to question (4) from Example \ref{freeex} (whence also all the others, using \cite{Oyono-Oyono:2009ua}) is `no'.   Thus if the answer to any of (1) to (4) is `yes', the family $(\Gamma_k)$ does not have property $(\tau)$; following the discussion on \cite[page 16]{Lubotzky:2004xw}, essentially using the fact that the congruence subgroups do have property $(\tau)$, this implies Serre's conjecture for $\Gamma$.
\end{proof}

Serre's conjecture is a deep and well-studied problem; for its current status, see \cite{Lubotzky:1997wq}.  A natural family of subgroups to study in this regard is the analogue of that in Example \ref{freeex}; indeed, it is not hard to see that if any family does not have property $(\tau)$ then this one cannot.
\end{serreex}

\section{Property A and the maximal Roe algebra}\label{asec}

In this section we prove Proposition \ref{aprop}.  We first recall one of the possible definitions of property A (cf.\ e.g.\ Theorem 3 in \cite{Brodzki:2007mi} or \cite{Tu:2001bs}).

\begin{propa}\label{propa}
$X$ is said to have \emph{property A} if for any $R,\epsilon>0$ there exists a map $\xi:X\to \ell^2(X)$ and $S>0$ such that:
\begin{itemize}
\item  for all $x,y\in{}X$, $\xi_x(y)\in[0,1]$;
\item for all $x\in{}X$, $\|\xi_x\|_2=1$;
\item  for all $x\in{}X$, $\xi_x$ is supported in $B(x,S)$;
\item if $d(x,y)\leq{}R$ then $|1-\langle{}\xi_x,\xi_y\rangle|<\epsilon$. 
\end{itemize}
\end{propa}

Given a map $\xi$ as in the above definition we associate:
\begin{itemize}
\item a `partition of unity' $\{\phi_y\}_{y\in{}X}$ defined by $\phi_y(x)=\xi_x(y)$;
\item a `kernel' $k:X\times{}X\to[0,1]$ defined by $k(x,y)=\langle{}\xi_x,\xi_y\rangle$;
\item a `Schur multiplier' $M_k^{alg}:\C_u[X]\to\C_u[X]$ defined by $(M_kT)(x,y)=k(x,y)T(x,y)$.
\end{itemize}

Now, the proof of \cite[Lemma 11.17]{Roe:2003rw} shows that for $T\in \C_u[X]\subseteq\mathcal{B}(\ell^2(X))$, one has the formula
\begin{displaymath}
M^{alg}_k(T)=\sum_{x\in{}X}\phi_xT\phi_x
\end{displaymath}
(convergence in the strong topology on $\mathcal{B}(\ell^2(X))$), whence the map $M_k^{alg}$ extends to a unital completely positive map from $C^*_u(X)$ to itself.  In our context, we would like to moreover extend $M_k^{alg}$ to $C^*_{u,max}(X)$; unfortunately, the formula above makes no sense here.  Instead we use the following simple lemma.

\begin{finlem}\label{finlem}
Let $\xi$ be as in Definition \ref{propa}, and $k$ the associated kernel.  Then for any $S_0>0$ there exists a finite collection of functions $\phi_i:X\to[0,1]$, $i=1,...,N$ such that for all $T\in\C_u^{S_0}[X]$,
\begin{displaymath}
M_k^{alg}(T)=\sum_{i=1}^N\phi_iT\phi_i.
\end{displaymath}
\end{finlem}

\begin{proof}
Let $S>0$ be a support bound for $\xi$ as in definition \ref{propa}.  Let $A_1$ be a maximal $S_1:=2S+S_0$-separated subset of $X$, and inductively choose $A_k$ to be a maximal $S_1$-separated subset of $X\backslash(A_1\cup...\cup{}A_{k-1})$.  Note that $A_n$ is empty for all $n$ suitably large: if not, there exist $x_n\in{}A_n$ for all $n$, and by maximality of each $A_k$, one has $d(x_n,A_{k})\leq2S_1$ for all $k<n$; hence $|B(x_n,2S_1)|\geq{}n-1$ for all $n$, contradicting bounded geometry of $X$.  Define now 
\begin{displaymath}
\phi_i=\sum_{x\in{}A_i}\phi_x,
\end{displaymath}
where $\{\phi_x\}$ is the partition of unity associated to $\xi$; the sum is taken in $l^\infty(X)$, converging with respect to the weak-$*$ topology.  Hence each $\phi_i$ as a well-defined element of $l^\infty(X)\subseteq\C_u[X]$.  Using the facts that $\langle{}\xi_x,\xi_y\rangle=0$ for $d(x,y)\geq{}2S$, and that for $(T(x,y))\in\C_u^{S_0}[X]$ one has that $T(x,y)=0$ for $d(x,y)\geq{}S_0$, it is routine to check that $\{\phi_i\}_{i=1}^N$ has the properties claimed.  
\end{proof}

\begin{extendlem}\label{extendlem}
Say that $\xi$ is an in Definition \ref{propa}.  Then the associated linear map $M_k^{alg}:\C_u[X]\to\C_u[X]$ extends to a unital completely positive (u.c.p.) map on any $C^*$-algebraic completion of $\C_u[X]$. 
\end{extendlem}
\begin{proof}
Let $T\in{}M_n(\C_u[X])$.  Then there exists $S_0>0$ such that all matrix entries of $T$ have propagation at most $S_0$.  Hence by Lemma \ref{finlem} there exist $\phi_i$, $i=1,...,N,$ such that if $\diag(\phi_i)\in{}M_n(\C_u[X])$ is the diagonal matrix with all non-zero entries $\phi_i$, then the matrix augmentation of $M_k^{alg}$, say $M_k^n$, is given by
\begin{displaymath}
M_k^n(T)=\sum_{i=1}^N\diag(\phi_i)T\diag(\phi_i);
\end{displaymath}
this implies that if $A$ is any $C^*$-algebraic completion of $\C_u[X]$, then the map
\begin{displaymath}
M_k^{alg}:\C_u[X]\to\C_u[X]\subseteq{}A
\end{displaymath}
is u.c.p. (for the order structure coming from $A$).  It is in particular then bounded, and thus extends to a u.c.p. map from $A$ to itself.
\end{proof}

Consider now the diagram
\begin{equation}\label{acommutes}
\vcenter{\xymatrix{ \C_u[X] \ar[r]^{M_k^{alg}} \ar[d] &  \C_u[X] \ar[d] \\
C^*_{u,max}(X) \ar[r]^{M_k^{max}} \ar[d]^\lambda & C^*_{u,max}(X) \ar[d]^\lambda \\
C^*_u(X) \ar[r]^{M_k^\lambda} & C^*_u(X), }}
\end{equation}
where the maps in the lower two rows are the u.c.p. maps given by the corollary.  The diagram commutes, as it does on the level of $\C_u[X]$ and as all the maps in the lower square are continuous.  The following simple lemma essentially completes the proof of Proposition \ref{aprop}.

\begin{closlem}\label{closlem}
For any $S>0$, $\C^S_u[X]$ is closed in $C^*_{u,max}(X)$.  In particular, for a Schur multiplier $M_k^{max}$ built out of a map $\xi$ with support parameter $S$ as in Definition \ref{propa}, the image of $M_k^{max}$ is contained in $\C^{2S}_u[X]$.
\end{closlem}

\begin{proof}
Say $(T_n)$ is a sequence in $\C^S_u[X]$ converging in the $C_{u,max}^*(X)$ norm to some $T$.  It follows that $\lambda(T_n)\to{}\lambda(T)$ in $C^*_u(X)$, whence $\lambda(T)$ is in $\C^S_u[X]$ (considered as a subalgebra of $C^*_u(X)$) by uniqueness of matrix representations on $\mathcal{B}(\ell^2(X))$, and moreover that the matrix entries of $T_n$ converge uniformly to those of $\lambda(T)$.  Lemma \ref{fundlem} now implies that the $T_n$ converge to $\lambda(T)$ (qua element of $\C^S_u[X]$) in the $C^*_{u,max}(X)$ norm.

The remaining comment follows as $M_k^{max}(\C_u[X])\subseteq\C_u^{2S}[X]$ under the stated assumptions.
\end{proof}

\begin{proof}[Proof of Proposition \ref{aprop}]
Using property A and Lemma \ref{fundlem}, there exists a sequence of kernels $(k_n)$ such that the associated Schur multipliers $M_{k_n}^\lambda$ and $M_{k_n}^{max}$ converge point-norm to the identity on $C^*_u(X)$ and $C^*_{u,max}(X)$ respectively.  Now, say $T\in{}C^*_{u,max}(X)$ is in the kernel of $\lambda$.  Hence $\lambda(M_{k_n}^{max}(T))=M^\lambda_{k_n}(\lambda(T))=0$ for all $n$, using commutativity of (\ref{acommutes}).  However, Lemma \ref{closlem} implies that $M_{k_n}^{max}(T)\in\C_u[X]$ for all $n$; as $\lambda$ is injective here, it must be the case that $M_{k_n}^{max}(T)=0$ for all $n$.  These elements converge to $T$, however, so $T=0$.  Hence $\lambda$ is injective as required.
\end{proof}

\section{The twisted uniform Roe algebra}\label{twist}

\begin{cembed}
Let $H$ be a real Hilbert space.  A map $f:X\to H$ is called a \emph{coarse embedding} if there exist non-decreasing functions $\rho_-,\rho_+:\R_+\to\R_+$ such that $\rho_-(t)\to\infty$ as $t\to\infty$ and for all $x,y\in X$
\begin{displaymath}
\rho_-(d(x,y))\leq\|f(x)-f(y)\|_H\leq\rho_+(d(x,y)).
\end{displaymath}
\end{cembed} 

From now on in this paper we fix a coarse embedding $f:X\to H$ of $X$ into a real Hilbert space $H$.  Passing to a subspace if necessary, we may assume moreover that $V:=\textrm{span}(f(X))$ is dense in $H$.

We rely heavily on material from the paper \cite{Yu:200ve}.  We introduce most of the objects from this paper as we need them, but for the reader's convenience we also give a list of notation in Appendix \ref{appendix2}; this notation is compatible with that from \cite{Higson:1999be}, which we will also often refer to.  We also introduce several new objects that are necessary at various stages in the proof; these too are included in the list in Appendix \ref{appendix2}.

\subsection{The twisted uniform Roe algebra}

In this section we introduce the maximal and reduced twisted uniform Roe algebras, which are (slightly simpler) variations on the twisted Roe algebra introduced by Yu in \cite[Section 5]{Yu:200ve}.  We then prove that the maximal and reduced versions are really the same.  

\begin{adef}\label{adef}
Denote by $V_a,V_b$ etc.\ the finite dimensional affine subspaces of $V$.  Let $V_a^0$ be the linear subspace of $V$ consisting of differences of elements of $V_a$.  Let $\Cliff_\C(V_a^0)$ be the complexified Clifford algebra of $V_a^0$ and $\mathcal{C}(V_a)$ the graded $C^*$-algebra of continuous functions vanishing at infinity from $V_a$ into $\Cliff_\C(V_a^0)$.  Let $\mathcal{S}$ be the $C^*$-algebra $C_0(\R)$, graded by even and odd functions, and let $\mathcal{A}(V_a)=\mathcal{S}\hat{\otimes}\mathcal{C}(V_a)$ (throughout, `$\hat{\otimes}$' denotes the graded spatial tensor product of graded $C^*$-algebras).   

If $V_a\subseteq V_b$, denote by $V_{ba}^0$ the orthogonal complement of $V_{a}^0$ in $V_b^0$.  One then has a decomposition $V_b=V_{ba}^0+V_a$ and corresponding (unique) decomposition of any $v_b\in V_b$ as $v_b=v_{ba}+v_{a}$.  Any function $h\in\mathcal{C}(V_a)$ can thus be extended to a multiplier $\tilde{h}$ of $\mathcal{C}(V_b)$ by the formula $\tilde{h}(v_b)=h(v_a)$.

Continuing to assume that $V_a\subseteq V_b$, denote by $C_{ba}:V_b\to\Cliff_\C(V_{ba}^0)$ the function $v_b\mapsto v_{ba}$ where $v_{ba}$ is considered as an element of $\Cliff_\C(V_{ba}^0)$ via the inclusion $V_{ba}^0\subseteq \Cliff_\C(V_{b}^0)$.  Let also $X$ be the unbounded multiplier of $\mathcal{S}$ given by the function $t\mapsto t$.  Define a $*$-homomorphism $\beta_{ba}:\mathcal{A}(V_a)\to\mathcal{A}(V_b)$ via the formula
\begin{displaymath}
\beta_{ba}(g\hat{\otimes}h)=g(X\hat{\otimes}1+1\hat{\otimes}C_{ba})(1\hat{\otimes}\tilde{h}),
\end{displaymath}  
where $g\in\mathcal{S}$, $h\in\mathcal{C}(V_a)$ and the term on the right involving $g$ is defined via the functional calculus for unbounded multipliers.

These maps turn the collection $\{\mathcal{A}(V_a)\}$ as $V_a$ ranges over finite dimensional affine subspaces of $V$ into a directed system.  Define the \emph{$C^*$-algebra of $V$} to be
\begin{displaymath}
\mathcal{A}(V)=\lim_\rightarrow\mathcal{A}(V_a)
\end{displaymath}
\end{adef}

\begin{wdef}\label{wdef}
Let $x$ be a point in $X$.  Define $W_n(x)$ to be the (finite dimensional) subspace of $V$ spanned by $\{f(y):d(x,y)\leq{}n^2\}$.  Write $\beta_n(x):\mathcal{A}(W_n(x))\to\mathcal{A}(V)$ for the map coming from the definition of $\mathcal{A}(V)$ as a direct limit.  
\end{wdef}

From now on, $\R_+\times H$ and its subset $\R_+\times V$ are assumed equipped with the weakest topology for which the projection to $H$ is continuous for the weak topology on the latter, and so that the function $(t,v)\mapsto{}t^2+\|v\|^2$ is continuous.  This topology makes $\R_+\times H$ into a locally compact Hausdorff space.

Note that the inclusion $\beta_{ba}:\mathcal{A}(V_a)\to\mathcal{A}(V_b)$ sends the central subalgebra $C_0(\R_+\times V_a)$ into $C_0(\R_+\times V_b)$ (here $C_0(\R_+)$ has been identified with the even part of $\mathcal{S}$). One has moreover that the limit of the corresponding directed system $(C_0(\R_+\times{}V_a),\beta_{ba})$ is $C_0(\R_+\times H)$ (where $\R_+\times H$ is equipped with the topology above).  

\begin{supp0}\label{supp0}
The \emph{support} of $a\in\mathcal{A}(V)$ is defined to be the complement of all $(t,v)\in\R_+\times H$ such that there exists $g\in C_0(\R_+\times H)$ such that $g(t,v)\neq0$ and $g\cdot a=0$.
\end{supp0}

The following lemma is very well-known; we record it as we will need it several times below.  It is a simple consequence of the fact that the norm of a normal element in a $C^*$-algebra is equal to its spectral radius.

\begin{invclo}\label{invclo}
Say $\mathcal{B}$ is a dense $*$-subalgebra of a $C^*$-algebra $B$.  Assume that the inclusion $\mathcal{B}\to B$ is \emph{isospectral} (i.e.\ an element of the unitisation $\tilde{\mathcal{B}}$ of $\mathcal{B}$ is invertible if and only if it is invertible in $\tilde{B}$).  Then any $*$-representation of $\mathcal{B}$ extends to $B$. \qed
\end{invclo}

We are now ready to give our variant of the \emph{twisted Roe algebra} from \cite{Yu:200ve}. 

\begin{tra}\label{tra}
The \emph{algebraic twisted uniform Roe algebra}, denoted $\C[X,\mathcal{A}]$, is defined as a set to consist of all functions $T:X\times X\to\mathcal{A}(V)$ such that
\begin{enumerate}
\item there exists $N\in\N$ such that for all $x,y\in X$, $T(x,y)\in\beta_N(x)(\mathcal{A}(W_N(x)))$, where $\beta_N(x)$ is as in definition \ref{wdef};
\item there exists $M>0$ such that $\|T(x,y)\|\leq{}M$ for all $x,y\in{}X$;
\item there exists $r_1>0$ such that if $d(x,y)>r_1$, then $T(x,y)=0$;
\item there exists $r_2>0$ such that $\textrm{support}(T(x,y))$ is contained in the open ball of radius $r_2$ about $(0,f(x))$ in $\R_+\times{}H$;
\item there exists $c>0$ such that if $Y\in\R\times{}W_N(x)$ has norm less than one and if $T_1(x,y)$ is such that $\beta_N(x)(T_1(x,y))=T(x,y)$, then the derivative of the function $T_1(x,y):\R\times W_N(x)\to\Cliff_\C(W_N(x))$ in the direction of $Y$, denoted $D_Y(T_1(x,y))$, exists, and $\|D_Y(T_1(x,y))\|\leq{}c$;
\item for all $x,y\in X$, $T(x,y)$ can be written as the image under $\beta_N(x)$ of a finite linear combination of elementary tensors from $\mathcal{A}(W_N(x))\cong\mathcal{S}\hat{\otimes}\mathcal{C}(W_N(x))$.
\end{enumerate}
$\C[X,\mathcal{A}]$ is then made into a $*$-algebra via matrix multiplication and adjunction, together with the $*$-operations on $\mathcal{A}(V)$.

Let now
\begin{displaymath}
E_X=\Big\{\sum_{x\in{}X}a_x[x]~\big|~a_x\in\mathcal{A}(V),\sum_{x\in X}{a_x^*a_x} \textrm{ is norm convergent}\Big\}.
\end{displaymath}
Equipped with the $\mathcal{A}(V)$-valued inner product and $\mathcal{A}(V)$-action given respectively by
\begin{displaymath}
\Big\langle{}\sum_{x\in{}X}a_x[x],\sum_{x\in{}X}b_x[x]\Big\rangle=\sum_{x\in X}a_x^*b_x \textrm{ and } \Big(\sum_{x\in X}a_x[x]\Big)a=\sum_{x\in X}a_xa[x],
\end{displaymath}
$E_X$ becomes a Hilbert $\mathcal{A}(V)$-module.  It is moreover equipped with a (faithful) representation of $\C[X,\mathcal{A}]$ by matrix multiplication.  The \emph{twisted uniform Roe algebra of $X$}, denoted $C^*_u(X,\mathcal{A})$, is defined to be the norm closure of $\C[X,\mathcal{A}]$ in $\mathcal{L}_{\mathcal{A}(V)}(E_X)$.

Finally, one defines the universal norm on $\C[X,\mathcal{A}]$ to be the supremum over all norms coming from $*$-representations to the bounded operators on a Hilbert space; it is well-defined by an analogue of Lemma \ref{fundlem}.  The \emph{maximal uniform twisted Roe algebra of $X$}, denoted $C^*_{u,max}(X,\mathcal{A})$, is the completion of $\C[X,\mathcal{A}]$ in this norm.
\end{tra} 

Note that there is a canonical quotient map
\begin{displaymath}
\lambda:C^*_{u,max}(X,\mathcal{A})\to C^*_u(X,\mathcal{A}).
\end{displaymath}

Having made these definitions, we are aiming in this section for the following result, which is stated without proof on page 235 of \cite{Yu:200ve}; nonetheless, we include a relatively detailed argument as it is necessary for the results of the rest of the paper and seemed to fit well within the philosophy of this piece.

\begin{mer}\label{mer}
The canonical quotient map $\lambda:C^*_{u,max}(X,\mathcal{A})\to C^*_{u}(X,\mathcal{A})$ is an isomorphism.
\end{mer}

The proof proceeds by showing that $\C[X,\mathcal{A}]$ is built up from simple parts for which the representation on $E_X$ is the same as the maximal one.  The following definition introduces the `parts' we use.

\begin{suppo}\label{supp}
The \emph{support} of $T\in \C[X,\mathcal{A}]$ is defined to be 
\begin{displaymath}
\{(x,y,u)\in{}X\times X\times(\R_+\times H)~|~u\in\textrm{support}(T(x,y))\}.
\end{displaymath} 
Let $O$ be an open subset $\R_+\times H$.  Define $\C[X,\mathcal{A}]_O$ to be the subset of  $\C[X,\mathcal{A}]$ whose elements have support in $X\times X \times O$, which is a $*$-algebra. Define $C^*_u(X,\mathcal{A})_O$ to be the closure of $\C[X,\mathcal{A}]_O$ in $C^*_u(X,\mathcal{A})$.
\end{suppo}

In the following, we will be dealing with various $*$-subalgebras of $\C[X,\mathcal{A}]$.  For any such $*$-algebra, the \emph{$E_X$ norm} is defined to be the norm it inherits as a subalgebra of $C^*_{u}(X,\mathcal{A})$, and the \emph{maximal norm} is the norm it inherits as a subalgebra of $C^*_{u,max}(X,\mathcal{A})$

\begin{disj}\label{disj}
For any $v\in V$ and $r>0$, write $B(v,r)$ for $\{(t,w)\in\R\times H~|~t^2+\|v-w\|^2<r^2\}$, the open ball in $\R\times H$ of radius $r$ about $(0,v)$.  Say $A$ is a subset of $X$ such that $B(f(x),r)\cap B(f(y),r)=\emptyset$ whenever $x,y\in A$ and $x\neq y$.
Let 
\begin{displaymath}
O=\cup_{x\in{}A}B(f(x),r).
\end{displaymath}
Then the maximal norm on $\C[X,\mathcal{A}]_{O}$ is (well-defined) and equal to the $E_X$-norm.
\end{disj}

\begin{proof}
As the balls $B(f(x),r)$ are assumed disjoint, a generic element of $\C[X,\mathcal{A}]_O$ looks like an $A$-parametrised sequence $\{T^x\}_{x\in{}A}$ such that each $T^x$ is an operator in $\C[X,\mathcal{A}]_{B(f(x),r)}$ and such that all the the $T^x$'s satisfy conditions (1) to (5) in Definition \ref{tra} for a uniform collection of constants; addition, multiplication and adjunction are then defined pointwise across the sequence. In other words, $\C[X,\mathcal{A}]_O$ is isomorphic to a certain $*$-subalgebra of the direct product $\Pi_{x\in A}\C[X,\mathcal{A}]_{B(f(x),r)}$.

Now, note that for any $T^x$ as above, the support conditions (3) and (4) from Definition \ref{tra} together with the condition that all $T^x(y,z)$ are supported in $B(f(x),r)$ implies that only finitely many `coefficients' $T^x(y,z)$ are non-zero.  It follows that $\C[X,\mathcal{A}]_O$ is an algebraic direct limit of $*$-algebras of the form
\begin{multline*}
M_n(\{\{b_x\}_{x\in{}A}\in\Pi_{x\in{}A}C_0(B(f(x),r)),\Cliff_\C(W_n(f(x))))~|~\{b_x\} \textnormal{ satisfies (2), (5)} \\\textnormal{from Definition \ref{tra} uniformly}\}).
\end{multline*}
The unitisation of each of these is inverse closed inside its norm closure in the unitisation of
\begin{displaymath}
M_n(\Pi_{x\in A}(C_0(B(f(x),r)),\Cliff_\C(W_n(f(x))))),
\end{displaymath}
however, whence Lemma \ref{invclo} implies that its maximal norm is the same as that induced from the representation on $E_X$.  Hence the same is true for $\C[X,\mathcal{A}]_O$.
\end{proof}

The following lemma is a special case of \cite[Lemma 6.3]{Yu:200ve}.

\begin{pastelem}\label{pastelem}
Say $X$ is decomposed as a disjoint union $X=A_1\sqcup...\sqcup A_n$.  Let $r>0$ and set
\begin{displaymath}
O_{r,j}=\bigcup_{x\in{}A_j}B(f(x),r).
\end{displaymath}
Then for any $r_0>0$ and any $k=1,...,n-1$ one has that
\begin{displaymath}
\lim_{r<r_0,r\to r_0}\C[X,\mathcal{A}]_{\cup_{j=1}^kO_{r,j}}+\lim_{r<r_0,r\to r_0}\C[X,\mathcal{A}]_{O_{r,k+1}}=\lim_{r<r_0,r\to r_0}\C[X,\mathcal{A}]_{\cup_{j=1}^{k+1}O_{r,j}}
\end{displaymath}
and
\begin{displaymath}
\lim_{r<r_0,r\to r_0}\C[X,\mathcal{A}]_{\cup_{j=1}^kO_{r,j}}\cap\lim_{r<r_0,r\to r_0}\C[X,\mathcal{A}]_{O_{r,k+1}}=\lim_{r<r_0,r\to r_0}\C[X,\mathcal{A}]_{(\cup_{i=1}^kO_{r,j})\cap O_{r,k+1}}.
\end{displaymath}
\end{pastelem}

\begin{proof}
In the set up of \cite[Lemma 6.3]{Yu:200ve}, set $i_0=1$, and $X_{1,j}=A_j$, $X'_{1,j}=A_{k+1}$ for each $j=1,..,k$.  Then in the notation of \cite{Yu:200ve}, $O_r=\cup_{j=1}^kO_{r,j}$, $O'_r=O_{r,k+1}$, whence our lemma follows.
\end{proof}

The next lemma is essentially lemma 6.7 from \cite{Yu:200ve}.

\begin{geomlem}\label{geomlem}
For any $r>0$ there exists a disjoint partition $X=A_1\sqcup...\sqcup A_n$ of $X$ such that for any $x,y\in{}A_j$, if $x\neq{}y$, then $B(f(x),r)\cap B(f(y),r)=\emptyset$. \qed
\end{geomlem}

\begin{proof}[Proof of proposition \ref{mer}]
Let 
\begin{displaymath}
O_r=\bigcup_{x\in X}B(f(x),r)\subseteq\R_+\times V;
\end{displaymath}
as one has that 
\begin{displaymath}
\C[X,\mathcal{A}]=\lim_{r\to\infty}\C[X,\mathcal{A}]_{O_r},
\end{displaymath}
it suffices to show that for each $\C[X,\mathcal{A}]_{O_r}$, and in fact for each $*$-algebra
\begin{displaymath}
\lim_{r<r_0,r\to r_0}\C[X,\mathcal{A}]_{O_r}
\end{displaymath} 
that the $E_X$-norm and maximal norm are the same.

Using Lemma \ref{geomlem}, for each $r>0$
\begin{displaymath}
O_r=\bigcup_{j=1}^n\Big(\bigsqcup_{x\in A_j}B(f(x),r)\Big);
\end{displaymath}
let $O_{r,j}=\sqcup_{x\in{}A_j}B(f(x),r)$.  Lemma \ref{disj} implies that the maximal norm on each $\C[X,\mathcal{A}]_{O_{r,j}}$ is the same as the $E_X$-norm, whence the same is true for
\begin{displaymath}
\lim_{r<r_0,r\to r_0}\C[X,\mathcal{A}]_{O_{r,j}}.
\end{displaymath}
Lemma \ref{pastelem} then implies that for each $k=1,...,n-1$, 
\begin{displaymath}
\lim_{r<r_0,r\to{}r_0}\C[X,\mathcal{A}]_{\cup_{j=1}^{k+1}O_{r,j}}
\end{displaymath}
is the pushout of
\begin{displaymath}
\lim_{r<r_0,r\to{}r_0}\C[X,\mathcal{A}]_{\cup_{j=1}^{k}O_{r,j}} \textrm{ and } \lim_{r<r_0,r\to{}r_0}\C[X,\mathcal{A}]_{O_{r,k+1}} 
\end{displaymath}
over
\begin{displaymath}
\lim_{r<r_0,r\to{}r_0}\C[X,\mathcal{A}]_{(\cup_{j=1}^{k}O_{r,j})\cap O_{r,k+1}};
\end{displaymath}
these pushouts exist in the category of $*$-algebras as the latter algebra is an ideal in the former two.  It follows that the same is true on the $C^*$-algebraic level (for either maximal or $E_X$ norms); as the $C^*$-norm on a pushout is uniquely determined by the norms on the other three $C^*$-algebras in the diagram defining it, this completes the proof.
\end{proof}

\begin{nuclearity} \label{nuclearity}
The proof of Lemma \ref{disj} also shows that $C^*_u(X,\mathcal{A})_{O}$ is type I for $O$'s as in Lemma \ref{disj}. The same argument as above then implies that in fact $C^*_u(X,\mathcal{A})$ itself is type I \cite[Chapter 6]{Pedersen:1979zr}, whence in particular nuclear \cite[Section 2.7]{Brown:2008qy}.
\end{nuclearity}

\begin{proof}[Proof of the remark.]
We elaborate a bit on the first sentence of the remark. Using bounded geometry of $X$, for a fixed $n$ there are only finitely many (up to isomorphism) vector spaces $W_n(f(x))$. Consequently, there are only finitely many non-isomorphic algebras $\mathrm{Cliff}_\C(W_n(f(x)))$, each of which is either a matrix algebra or a direct sum of two matrix algebras. Hence the completion of the $*$-algebra in the display in the proof of Lemma \ref{disj} is isomorphic to a finite direct sum of matrix algebras over commutative algebras.  The fact that $C^*_u(X,\mathcal{A})$ is type I now follows from this, the proof of Proposition \ref{mer}, and the fact that the type I property is preserved under quotients, direct sums and limits.
\end{proof}

\section{Proof of Theorem \ref{main}}\label{mainproof}

Throughout the remainder of the paper, we will mainly be dealing with $K$-theoretic arguments.  As it is more convenient for what follows, we will use $K$-theory for graded $C^*$-algebras; see Appendix \ref{appendix1} for our conventions in this regard.  Also in this appendix, we record our precise conventions concerning composition of $E$-theory classes / asymptotic morphisms.  This is important, as many of the $C^*$-algebras we use are not separable (or even $\sigma$-unital), so some of the methods for composing asymptotic morphisms in the literature (for example, that from \cite{Connes:1990kx} and \cite[Appendix 2.B]{Connes:1994zh}) may fail. 

\subsection{The Bott and Dirac morphisms}\label{bdsec}

In \cite[Section 7]{Yu:200ve}, Yu defines Dirac and Bott morphisms, and shows that they define asymptotic morphisms
\begin{displaymath}
\alpha_t:C^*(P_d(X),\mathcal{A}) \leadsto C^*(P_d(X),K) ~~\textnormal{ and } ~~\beta_t:\mathcal{S}\hat{\otimes}C^*(P_d(X))\leadsto C^*(P_d(X),\mathcal{A})
\end{displaymath}
respectively.  These are constructed on the algebraic level by applying the `local' versions $\theta^n_t(x)$ and $\beta(x)$ described in Appendix \ref{appendix2} to matrix entries, plus a rescaling in the case of $\beta$. Namely, 
\begin{equation}\label{dirac}
\alpha_t(T)(x,y)=(\theta^N_t(x))(T_1(x,y)),
\end{equation}
where $N$ is such that for all $x,y\in X$ there exists $T_1(x,y)\in\mathcal{A}(W_N(x))$ such that $\beta_N(T_1(x,y))=T(x,y)$; and 
\begin{equation}\label{bott}
\beta_t(g\hat\otimes T)(x,y)=(\beta(x))(g_t)\cdot T(x,y).  
\end{equation}

Now, these formulas also make sense in our context, and we use them to define Dirac and Bott morphisms on the algebraic level precisely analogously.  However, our Dirac and Bott morphisms do not have the same domains and ranges as Yu's; precisely, ours will look like
\begin{displaymath}
\alpha_t:C_{u,\xx}^*(X,\mathcal{A}) \leadsto UC_{\xx}^{*,g}(X) \textnormal{ and } \beta_t:\mathcal{S}\hat{\otimes}C^*_{u,\xx}(X)\leadsto C^*_{u,\xx}(X,\mathcal{A})
\end{displaymath}
($UC^{*,g}_{\xx}(X)$ is defined below).  Our aim in this Subsection is to show that the formulas in lines (\ref{dirac}) and (\ref{bott}) above really do define asymptotic morphisms between the algebras claimed.  

We start by defining the range of the Dirac morphism.  Throughout this section $\mathcal{H}=\mathcal{H}_0\oplus \mathcal{H}_1$ denotes a graded Hilbert space with separable and infinite-dimensional even and odd parts $\mathcal{H}_0$, $\mathcal{H}_1$ respectively.  For later computations, it will be useful to make the particular choice of $\mathcal{H}$ described in Appendix \ref{appendix2}.

\begin{ucalgdef}\label{ucalgdef}
Let $\mathcal{K}(\mathcal{H}_0)$ denote the (trivially graded) copy of the compact operators on $\mathcal{H}_0$.  Define $U\C_0[X]$ to be the $*$-algebra of finite propagation $X$-by-$X$ matrices $(T(x,y))_{x,y\in X}$ such that there exists $N\in\N$ such that for all $x,y\in X$, $T(x,y)$ is an operator in $\mathcal{K}(\mathcal{H}_0)$ of rank at most $N$.  $U\C_0[X]$ is called the \emph{(trivially graded) $*$-algebra of uniform operators}.  Define $U\C[X]$ to be the $*$-algebra of $X$-by-$X$ matrices with entries in $\mathcal{K}(\mathcal{H}_0)$ such that for all $\epsilon>0$ there exists $T_0\in U\C_0[X]$ with $\|T_0(x,y)-T(x,y)\|<\epsilon$ for all $x,y\in X$. 

$U\C[X]$ is represented naturally by matrix multiplication on $\ell^2(X)\otimes \mathcal{H}_0$.  Its norm closure in the associated operator norm is called the \emph{(trivially graded) uniform algebra of $X$}, and denoted $UC^{*}(X)$.  The closure of $U\C[X]$ for the maximal norm is called the \emph{(trivially graded) maximal uniform algebra of $X$} and denoted $UC^{*}_{max}(X)$ (it is well-defined by an obvious analogue of Lemma \ref{fundlem}).

We define non-trivially graded versions $U\C_0^g[X]$, $U\C^g[X]$, $UC^{*,g}(X)$ and $UC^{*,g}_{max}(X)$ precisely analogously by using $\mathcal{K}(\mathcal{H})$ for the matrix entries, and decreeing that an element $T\in U\C^g[X]$ is even (odd) if and only if all its matrix entries are even (odd).
\end{ucalgdef}

The difference between $UC^{*,g}_{\xx}(X)$ and $UC^*_{\xx}(X)$ is not major: one has of course that there is an isomorphism
\begin{equation}\label{grademap}
\phi:UC^{*,g}_\xx(X)\to UC^*_\xx(X)\hat{\otimes}\textnormal{Cliff}_\C(\R^2)
\end{equation}
whence the following lemma, which we record for later use.

\begin{gradelem}\label{gradelem}
An isomorphism $\phi$ as in line (\ref{grademap}) above induces an isomorphism
\begin{displaymath}
\phi_*:K_*(UC^{*,g}_\xx(X))\to K_*(UC^*_\xx(X)).
\end{displaymath}
Moreover, if $\iota:UC^*_\xx(X)\to UC^{*,g}_\xx(X)$ is the inclusion induced by the inclusion $\mathcal{H}_0\hookrightarrow\mathcal{H}$, then $\phi_*$ is the inverse to $\iota_*$.  In particular, $\phi_*$ does not depend on the choice of $\phi$.  \qed
\end{gradelem}

\begin{ucalgrem}
The distinction between $U\C_0[X]$ and $U\C[X]$ is also not major: the two have the same representation theory, using the fact that
\begin{displaymath}
\{f\in\ell^\infty(X,\mathcal{K})~|~\text{ there exists }n\in\N \text{ such that for all $x\in X$,}\text{ rank}(f(x))\leq n\}
\end{displaymath}
is isospectrally included inside its norm closure in $\ell^\infty(X,\mathcal{K})$, the argument of Lemma \ref{fundlem}, and Lemma \ref{invclo}.  The two are convenient for slightly different purposes however, so we include both.  Similar comments apply in the graded case.
\end{ucalgrem}

\begin{ucalglem}\label{ucim}
Let $T\in \C[X,\mathcal{A}]$ be such that there exists $N$ such that for all $x,y\in X$ there exists $T_1(x,y)\in\mathcal{A}(W_N(x))$ such that $\beta_N(T_1(x,y))=T(x,y)$.  Then for each $t\in[1,\infty)$, the $X$-by-$X$ matrix with $(x,y)^{\rm th}$ entry
\begin{displaymath}
\theta^N_t(x)(T(x,y))
\end{displaymath}
is an element of $U\C^g[X]$.
\end{ucalglem}

\begin{proof}
Note that $(\theta^N_t(x)(T(x,y)))_{x,y\in X}$ is a finite propagation matrix with uniformly bounded entries in $\mathcal{K}(\mathcal{H})$; all we need show therefore is that the entries can be approximated by elements of $\mathcal{K}(\mathcal{H})$ of uniformly finite rank. 

Let $M,r_1,r_2,c$ be constants with respect to which $T$ satisfies the conditions in Definition \ref{tra} parts $2$, $3$, $4$, $5$ respectively.  Let $x_0\in X$ be such that $W_N(x_0)$ is of maximal dimension (such exists as $X$ has bounded geometry).  

Define $S(x_0)$ to be the subset of $h\in\mathcal{A}(W_N(x))$ such that: 
\begin{itemize}
\item $h$ can be written as a finite sum of elementary tensors of elements from $\mathcal{S}$ and $\mathcal{C}(W_N(x_0))$;
\item $\|h\|\leq M$;
\item support$(h)$ is contained in the ball of radius $r_2$ about $(0,f(x_0))$;
\item for all $Y\in\R\times W_n(x_0)$ of norm less than one, $\|D_y h\|\leq c$
\end{itemize}
(in other words, $S(x_0)$ consists of elements of $\mathcal{A}(W_N(x))$ satisfying all of the conditions on the $T(x,y)$).  Note that $S(x_0)$ is precompact by Arzela-Ascoli, whence its image under the continuous map $\theta^N_t(x_0)$,
\begin{displaymath}
\theta^N_t(x_0)(S(x_0))\subseteq \mathcal{K}(\mathcal{H}),
\end{displaymath}
is also precompact.  In particular, for any $\epsilon>0$ there exists $P>0$ such that any $K\in \theta^N_t(x_0)(S(x_0))$ can be approximated within $\epsilon$ by a rank $P$ operator.

Note finally that the group of linear isometries of $H$ acts on both $\mathcal{A}(V)$ and $\mathcal{K}(\mathcal{H})$ by $*$-au\-to\-mor\-phisms and that the maps $\theta^N_t(x)$ are `equivariant' in the sense that if $A:H\to H$ is a linear isometry taking $W_N(x)$ into $W_N(x_0)$ then $\theta_t^N(x_0)(A(h))=A(\theta_t^N(x)(h))$ for all $h\in\mathcal{A}(W_N(x))$.  In particular, all images of all the elements $T(x,y)\in\mathcal{A}(W_n(x))$ for our original $T$ are contained in unitary conjugates of $\theta^N_t(x_0)(S(x_0))$; by the previous paragraph, this proves the lemma.
\end{proof}

As it simplifies certain arguments, we define a dense $*$-subalgebra of $\mathcal{S}$ by setting
\begin{equation}\label{s0}
\mathcal{S}_0=\{f\in \mathcal{S}~|~f \text{ is compactly supported and continuously differentiable}\}.
\end{equation}

One now needs to show the following claims:
\begin{ashomo}\label{ashomo}
The families of maps 
\begin{displaymath}
\alpha_t:\C[X,\mathcal{A}]\leadsto U\C[X] \textnormal{ and } \beta_t:\mathcal{S}_0\hat\otimes_{alg}\C_u[X]\leadsto \C[X,\mathcal{A}]
\end{displaymath}
defined by the formulas in lines (\ref{dirac}) and (\ref{bott}) above are asymptotically well-defined\footnote{i.e.\ the choices made in the definitions do not affect the resulting maps $\alpha:\C[X,\mathcal{A}]\to \mathfrak{A}(U\C[X])$ and $\beta:\mathcal{S}\hat\otimes\C_u[X] \to \mathfrak{A}(\C[X,\mathcal{A}])$, and the images of each map are where we claim they are.}, and define asymptotic families. 
\end{ashomo}

\begin{completions}\label{completions}
Both of the families of maps in Claim \ref{ashomo} above extend to the $C^*$-algebra completions, for both $red\to red$ and $max\to max$ versions.
\end{completions}

\begin{proof}[Proof of Claim \ref{ashomo}]
As the proofs are essentially the same as those in \cite{Yu:200ve}, we only give a relatively short summary.

We first check that the formulas defining $\alpha_t$, $\beta_t$ make sense, and do not depend on any of the choices involved.  Note that all of the unbounded operators involved in the definition of the various $\theta$s and $\beta$s we need are essentially self-adjoint by the results of \cite{Higson:1999be}, hence the functional calculi needed make sense. Furthermore, there is a choice of $N$ involved in the definition of $\alpha_t$. As pointed out in \cite[page 230]{Yu:200ve}, this choice asymptotically does not matter, by \cite[proof of Proposition 4.2]{Higson:1999be}. \footnote{For the reader's convenience (it will not be used in what follows), we note the difference between the approach from \cite{Yu:200ve} (using the operators $B_{n,t}$) and the one in \cite{Higson:1999be} (using direct limits). In the latter, the morphisms analogous to our $\theta_t^N(x)$s are defined to land in $\mathcal{S}\hat\otimes \K(\mathcal{H}_{W_N(x)})$ (and not in $\K(\mathcal{H})$ as here) and the direct limit of $\mathcal{S}\hat\otimes \K(\mathcal{H}_{W_N(x)})$s that one needs to take is not isomorphic to $\mathcal{S}\hat\otimes \K(\mathcal{H})$ (cf.\ \cite[Remark on p.18]{Higson:1999be}).}

Next, we check that the ranges of $\alpha_t$, $\beta_t$ are in $U\C^g[X]$, $\C[X,\mathcal{A}]$ respectively.  We did this for $\alpha_t$ in Lemma \ref{ucim}.
Considering $\beta_t$, note that because of the conditions on $\mathcal{S}_0$ and $\C_u[X]$, the image of some $g\hat\otimes T\in\mathcal{S}_0\hat\otimes \C_u[X]$ under $\beta_t$ satisfies the conditions of Definition \ref{tra}; this essentially follows from the fact that all of the $\beta(x)(g_t)$s are translates of each other\footnote{The idea here is the same as that in Lemma \ref{ucim}, as one might expect, but a little simpler.}.\\

We now show that $\alpha_t$, $\beta_t$ define asymptotic families. It is argued in \cite[Proof of Lemma 7.2]{Yu:200ve} that given $R\geq0$ and the parameters of an element $T\in \C[X,\mathcal{A}]$, the norm--error resulting from switching $x$ to $y$ 
in the formula for $\alpha_t(T)$ (i.e.\ $\|\theta_t^N(x)-\theta_t^{N}(y)\|$) tends $0$ as $t\to\infty$ uniformly over $(x,y)$ in any fixed entourage. Furthermore, by \cite[Lemma 7.5]{Yu:200ve}, $\|(\theta_t^N(x))(ab)-(\theta_t^N(x))(a)\cdot(\theta_t^N(x))(b)\|$ tends to $0$ uniformly in $x\in X$.   Note finally that using bounded geometry of $X$, given two finite propagation matrices, there is a uniform bound on the number of summations and multiplications involved in each entry of their product. These facts, together with (as usual) Lemma \ref{fundlem}, yield that the $\alpha_t$ define an asymptotic family. 

To show that $\beta_t$ defines an asymptotic family, we argue as in the previous paragraph. In this case, each $\beta_t(x)$ is a $*$-homomorphism, so we only need to show that the norm $\|\beta(x)(g_t)-\beta(y)(g_t)\|$ goes to $0$ uniformly on each entourage of the form $\{(x,y)\in X\times X\mid d(x,y)\leq R\}$ for fixed $g\in\mathcal{S}_0$ and $R\geq 0$. First, we reduce to the finite-dimensional situation. Let $W$ be a finite-dimensional subspace of $V$, which contains $0$, $f(x)$ and $f(y)$. Then $\beta(x)=\beta_{V,W}\circ\beta_{W,x}$ and $\beta(y)=\beta_{V,W}\circ\beta_{W,y}$, where $\beta_{V,W}:\mathcal{A}(W)\to\mathcal{A}(V)$ is the $*$-homomorphism associated to the inclusion $W\subset V$ and $\beta_{W,x},\beta_{W,y}:\mathcal{S}\to\mathcal{A}(W)$ are associated respectively to the inclusion of the zero-dimensional affine subspaces $\{f(x)\}$ and $\{f(y)\}$ to $W$. Thus it suffices to deal with $\beta_{W,x}$ and $\beta_{W,y}$.

Denote by $C_x$ and $C_y$ the Clifford multipliers used to define $\beta_{W,x}$ and $\beta_{W,y}$, that is, $C_x(v)=v+f(x)$ and $C_y(v)=v+f(y)$ as functions on $W$. The main point to observe is that $C_x-C_y$ is a bounded multiplier of $\mathcal{C}(W)$, namely by the constant function $v\mapsto f(x)-f(y)$. In particular, its norm is at most some $R'$, which exists, and depends only on $R$, as $f$ is a coarse embedding. The remainder of the proof is a standard argument: We first prove an estimate for $\|\beta_{W,x}(g_t)-\beta_{W,y}(g_t)\|$ when $g(s)=\frac1{s+i}$\footnote{This function is not in $\mathcal{S}_0$, of course, but this is not important.}. Denote $G_x=X\hat\otimes 1+1\hat\otimes C_x$ and $G_y=X\hat\otimes 1+1\hat\otimes C_y$. Then
\begin{align*}
  \|\beta_{W,x}(g_t)-\beta_{W,y}(g_t)\|&=\left\|\frac1{\frac1tG_x+i}-\frac1{\frac1tG_y+i}\right\|\leq\\
  &\leq \left\|\frac{\sqrt{t}}{G_x+ti}\right\|\|G_x-G_y\|\left\|\frac{\sqrt{t}}{G_y+ti}\right\|\leq
  \frac1{\sqrt{t}}\cdot R'\cdot \frac1{\sqrt{t}}=\frac{R'}{t},
\end{align*}
since $|\frac{\sqrt{t}}{s+ti}|\leq\frac1{\sqrt{t}}$ for all $s\in\R$. The same argument also works for $g(s)=\frac1{s-i}$. However, by the Stone--Weierstrass theorem, the algebra generated by these two functions is dense in $\mathcal{S}$, so we are done.
\end{proof}

\begin{proof}[Proof of Claim \ref{completions}]
For the $max\to max$ versions, we just observe that by the previous Claim, we have a $*$-homomorphism from the algebraic version of the corresponding left-hand side into the asymptotic algebra of the maximal completion of the right-hand side. Hence by the universality of the maximal completion, we are essentially done.  The only potential problem occurs in the case of $\beta$, and is that the algebraic version on the left-hand side is in fact $\mathcal{S}_0\hat\otimes_{\rm alg}\C_u[X]$; we thus need to see that the maximal completion of this algebra is $\mathcal{S}\hat\otimes C^*_{u,max}(X)$.  This follows from the fact that we can `restrict' any $*$-representation of $\mathcal{S}\hat\otimes_{\rm alg}\C_u[X]$ to a $*$-representation of $\C_u[X]$ (see the proof of \cite[Theorem 3.2.6]{Brown:2008qy}), and of course nuclearity of $\mathcal{S}$.

We can argue the $red\to red$ extension of $\alpha$ by the same argument (except that we complete the right-hand side in the reduced norm), plus Proposition \ref{mer}. The existence of the $red\to red$ extension of $\beta$ is proved in \cite[Lemma 7.6]{Yu:200ve}. We briefly summarise that argument, which also goes through in our case.

First it is proved that $\|\beta_t(g\hat\otimes T)\|\leq\|g\|\|T\|$ for all $g\in \mathcal{S}_0$, $T\in \C_u[X]$. Thanks to this, one can extend $\beta_t$ continuously to elements of the type $g\hat\otimes T$, where $g\in\mathcal{S}$ and $T\in C^*_u(X)$. Next, extend $\beta_t$ by linearity to $\mathcal{S}\hat\otimes_{alg} C^*_u(X)$.  The whole asymptotic family $(\beta_t)_{t\in[1,\infty)}$ can at this stage be considered as a $*$-representation of the (graded) algebraic tensor product of two $C^*$-algebras into the asymptotic algebra of $C^*_u(X,\mathcal{A})$, whence it extends to the maximal tensor product of them. The argument is finished by employing nuclearity of $\mathcal{S}$.
\end{proof}

\subsection{Morita equivalence of uniform algebras and uniform Roe algebras}\label{moritasec}

In this subsection, we build a Morita equivalence between $C^*_{u,\xx}(X)$ and $UC^*_{\xx}(X)$, and show that its inverse on the level of $K$-theory is given by a certain $*$-homomorphism. This is a general result which appears to be of some interest in its own right.  

\begin{morita}\label{prop:morita-eq}
The $C^*$-algebras $UC^*_{\xx}(X)$ and $C^*_{u,\xx}(X)$ are strongly Morita equivalent.
\end{morita}

Note that the result applies to the \emph{trivially} graded versions of the uniform algebra; in Section \ref{uniformcase} we will use Lemma \ref{gradelem} to relate this back to the non-trivially graded versions.

Let us remark that the $C^*$-algebras $UC^*_{\xx}(X)$ are not $\sigma$-unital if $X$ is infinite, so in this case strong Morita equivalence is not equivalent to stable isomorphism; and indeed, $UC^*_{\xx}(X)$ and $C^*_{u,\xx}(X)$ are not stably isomorphic.

\begin{proof}
Recall that part of the data used to construct $U\C_0[X]$ is a trivially graded Hilbert space $\mathcal{H}_0$.  We construct a pre-Hilbert module $E_{alg}$ over $\C_u[X]$ as follows.  Let $E_{alg}$ be the set of all finite propagation $X$-by-$X$ matrices with uniformly bounded entries in $\mathcal{H}_0$; in a departure from our usual convention, we write $\xi_{xy}$ for the $(x,y)^\textnormal{th}$ entry of an element of $E_{alg}$.  $E_{alg}$ is then a vector space, and we define a $\C_u[X]$-valued inner product $(\cdot|\cdot)$ on it by `multiply the transpose of the first matrix with the second matrix and take the $\mathcal{H}_0$-inner products where appropriate'. Formally, put  
\begin{displaymath}
\left((\xi_{xy})\mid(\eta_{xy})\right)_{zw}=\sum_y\langle   \xi_{yz}|\eta_{yw}\rangle.
\end{displaymath}
Using bounded geometry of $X$ and the Cauchy--Schwartz  inequality, the result is a complex matrix with finite propagation and uniformly bounded entries, hence an element of $\C_u[X]$. It is easy to check that $(\cdot|\cdot)$ is right-$\C_u[X]$-linear (note that $\langle\cdot|\cdot\rangle$ on $\mathcal{H}_0$ is assumed to be linear in the second variable). It is also positive-definite, since for any $\xi\in E_{alg}$ and fixed $x\in X$, the (effectively finite) matrix $\left(\xi_{xz}|\xi_{xw}\right)_{z,w\in X}$ is positive-definite by a standard argument, thus so is the locally finite sum $\sum_{x\in X}\left(\xi_{xz}|\xi_{xw}\right)_{z,w\in X}=(\xi|\xi)$. The fact that $(\xi|\xi)=0$ implies $\xi=0$ is now also clear.

Through the usual process of simultaneous completion of the $*$-algebra $\C_u[X]$ and pre-Hilbert module $E_{alg}$ \cite[pages 4-5]{Lance:1995ys}, we obtain Hilbert $C^*$-modules $E$ and $E_{max}$ over $C^*_u(X)$ and $C^*_{u,max}(X)$ respectively.

We next show that both $E$ and $E_{max}$ are full (that is, the closure of $(E_{\xx}|E_{\xx})$ is $C^*_{u,\xx}(X)$) by showing that $E_{alg}$ is `full', i.e.\ that $(E_{alg}|E_{alg})=\C_u[X]$. Any operator in $\C_u[X]$ can be written as a finite sum of   operators of the type $\zeta\cdot t$, where $\zeta\in \ell^\infty(X)$ and $t$ is a partial translation (we can even require the partial translations involved in the decomposition to be orthogonal), see e.g.\ \cite[Lemma 4.10]{Roe:2003rw}. Choosing any unit vector $v\in \mathcal{H}_0$, we define $\zeta v\in E_{alg}$ by $(\zeta v)_{yy}=\zeta(y)v$ and $(\zeta v)_{xy})=0$ if $x\not=y$. Similarly we put $tv\in E_{alg}$ to be the matrix of $t$, where we replace each one by $v$ and each zero by $0\in \mathcal{H}_0$. Now clearly $(\zeta v|tv)=\zeta\cdot t\in \C_u[X]$.

We now identify the `finite--rank' (in the sense of Hilbert $C^*$-module theory) operators  on $E_{alg}$ with $U\C_0[X]$. Denote by $\pi:U\C_0[X]\to\B(E_{alg})$ the $*$-homomorphism which is best described as `multiply the matrices, using the $\mathcal{K}(\mathcal{H}_0)$ module-structure on $\mathcal{H}_0$ to multiply entries'. The formula for $\pi(T)\xi$, where $T=(T(x,y))$ and $\xi=(\xi_{xy})$, is
$$
\big(\pi(T)(\xi)\big)_{xy} = \sum_{z}T(x,z)(\xi_{zy}).
$$
It is immediate that $\pi$ maps $U\C_0[X]$ into the finite rank operators on $E_{alg}$: any $T\in U\C_0[X]$ is a finite sum of operators with finite propagation and rank one entries; and each such acts as a `rank one' operator on $E$. Moreover, each `rank one' operator is in the image of $\pi$. We now show that $\pi$ is in fact injective on $U\C_0[X]$. Given a nonzero operator $T=(T(x,y))_{x,y\in X}\in U\C_0[X]$, there is a vector $\eta=(\eta_x)_{x\in X}\in \ell^2(X,\mathcal{H}_0)$, such that $T\eta\not=0$. Considering $\eta$ as a diagonal vector in $E_{alg}$, the diagonal of $(\pi(T)\eta|\pi(T)\eta)\in \C_u[X]$ is precisely $x\mapsto \|[\pi(T)\xi]_x\|^2$, hence non-zero. Altogether, we have shown that the $*$-algebra $U\C_0[X]$ is isomorphic to the $*$-algebra of `finite--rank' operators on $E_{alg}$. 

We now address the reduced case. Denote by $\lambda:\C_u[X]\to \B(\ell^2(X))$ the usual representation (so by completing $\C_u[X]$ and $E_{alg}$ in this norm we obtain a Hilbert module $E$ over $C^*_u(X)$).  Let $H'=E\otimes_\lambda\ell^2(X)$. The Hilbert space $H'$ carries a faithful representation of $\K(E)$; we will identify it with $UC^*(X)$. First, note that $H'\cong \ell^2(X,\mathcal{H}_0)$.  Indeed, any vector in $E_{alg}$ can be written as a sum of matrices whose support (in $X\times X$) is a partial translation. Moving this partial translation across $\otimes_\lambda$, it follows that any simple tensor $\xi\otimes\eta\in E_{alg}\otimes_\lambda \ell^2(X)$ can we written in a form where $\xi$ is supported on the diagonal in $X\times X$. Noting that we can also slide $\ell^\infty (X)$ functions over $\otimes_\lambda$, we conclude that $H'\cong \ell^2(X,\mathcal{H}_0)$. It is now a straightforward calculation to show that the elements of $U\C_0[X]$ act on $H'$ precisely as in the usual representation, whence $UC^*(X)\cong\K(E)$.

It remains to deal with the maximal case. The claim is that if we complete $E_{alg}$ to $E_{max}$ with the norm induced from $C^*_{u,max}(X)$, we get that $\pi$ extends to an isomorphism from $UC^*_{max}(X)$ to $\K(E_{max})$. This is a general argument; so to keep the notation simple, we shall assume that we are given two $*$-algebras $A$ and $B$, together with a right pre-Hilbert module $F$ over $A$, which is at the same time a left pre-Hilbert module over $B$, and is such that the inner products are compatible and full. This implies that the algebra of $A$--finite--rank operators on $F$ is isomorphic to $B$, and similarly for the left structure. In what follows, we shall always refer to the right module structures, and think of $B$ as being the finite--rank operators. Finally, we assume that both $A$ and $B$ have maximal $*$-representations.

We use the fact that choosing a $C^*$-norm on $A$ determines a Hilbert module completion $\overline F$ of $F$ (with the norm on $F$ given by the formula $\|\xi\|^2=\|(\xi|\xi)\|_A$) and hence a $C^*$-norm on $B$, sitting as a dense subalgebra in $\K(\overline F)$. Since the norm of elements $T$ of $B$ is determined by the formula $\|T\|^2=\sup_{\xi\in F,\|\xi\|\leq 1}\|(T\xi|T\xi)\|_A$, it is clear that the inequalities between norms on $A$ yield the same inequalities between the norms induced on $B$. Furthermore, we can induce also from $B$ to $A$ using the conjugate module $F^*$. Finally, the pre-Hilbert module $F^*\otimes_BF$ is isomorphic to the finite--rank operators on $F^*$ (as a pre-Hilbert bimodule over itself), which is in turn isomorphic to $A$; and similarly $F\otimes_AF^*\cong B$ as pre-Hilbert bimodules. Denoting the completions by bars, note that $F^*\otimes_BF$ is dense not only in $\overline{F^*\otimes_BF}_{\bar A}\cong \bar A_{\bar A}$, but also in $\overline{F^*}\otimes_{\bar B}\overline{F}_{\bar A}$ (where $\bar B\cong \K(\overline{F})$, which in turn determines the completion $\overline{F^*}$ of $F^*$). On elements $z\in F^*\otimes_BF$, however, both norms are determined by the same formula $(z|z)=(y|Xy)\in A$, where $z=\sum_{i=1}^kx_i\otimes y_i$, $y=(y_1,\dots,y_k)^t\in F^k$, $X\in M_k(B)$ is the positive matrix whose $(i,j)$-entry is $(x_i|x_j)$. Consequently, $\overline{F^*}\otimes_{\bar B}\overline{F}_{\bar A}\cong \bar A_{\bar A}$. It follows that the process described here of inducing norms from $A$ to $B$ and conversely is an isomorphism of lattices of $C^*$-norms on $A$ and $B$. Thus, the maximal norms correspond to each other, and we are done.
\end{proof}

Now, by the results of Exel \cite{Exel:1993pt}, a strong Morita equivalence bimodule $E$ induces a homomorphism $E_*$ on $K$-theory, which is in fact an isomorphism, without any assumptions on separability or $\sigma$-unitality; we thus have isomorphisms
\begin{displaymath}
E_*:K_*(UC^*_{\xx}(X))\cong K_*(C^*_{u,\xx}(X))
\end{displaymath}
(we should write `$(E_\xx)_*$' in the above, but prefer to keep the notation uncluttered).
The inverses to these isomorphisms can of course be described by the dual bimodules to $E$ and $E_{max}$; however, it will be useful to have a different description of the inverse, using the maps induced on $K$-theory by any one of a certain family of $*$-homomorphisms.  Our next step is to describe this family.

For each $x\in X$, choose a unit vector $\eta_x\in \mathcal{H}_0$. In a sense, we have chosen a unit vector in each `copy of $\mathcal{H}_0$' in $\ell^2(X,\mathcal{H}_0)$, thus designating a copy of $\ell^2(X)$ inside $\ell^2(X,\mathcal{H}_0)$. This induces an injective $*$-homomorphism $i_P: \C_u[X]\to U\C[X]$, explicitly defined below.

Define $e(x,y)=\theta_{\eta_x,\eta_y}\in\B(\mathcal{H}_0)$. The operators $E(x,y)\in\B(\ell^2(X,\mathcal{H}_0))$ that have $e(x,y)$ at $(x,y)^\textnormal{th}$-entry and $0$s elsewhere can be thought of as matrix units; and for $T=(T(x,y))\in \C_u[X]$, $[i_P(T)](x,y)=T(x,y)E(x,y)$. Furthermore, for $x\in X$, let $P_x:=\theta_{\eta_x,\eta_x}\in \B(\mathcal{H}_0)$ be the rank-one projection onto $\Span\{\eta_x\}$. Denote by $P=\diag(P_x)\in\ell^\infty(X,\mathcal{K})$ the `diagonal operator', with $P_x$s on the diagonal at the $(x,x)$-entries. Then $i_P(1)=P$.

Note that by lemma \ref{fundlem}, $i_P$ extends to give $*$-homomorphisms
$$
i_P:C^*_{u,\xx}(X)\to UC^*_{\xx}(X).
$$

\begin{kmap}\label{kmap}
For any choice of $i_P$ as described above, $i_{P*}:K_*(C^*_{u,\xx}(X))\to K_*(UC^*_\xx(X))$ is the inverse isomorphism to $E_*$ (in the sense of Exel \cite{Exel:1993pt}). 
\end{kmap}

\begin{proof}
Recall from \cite{Exel:1993pt} the way in which a Morita equivalence $A-B$-bimodule $E$ induces a map $K_0(A)\to K_0(B)$.  First, it is proved that one may represent elements of $K_0(A)$ as Fredholm operators $F$ (in the appropriate sense) between two (right) Hilbert $A$-modules $M$, $N$ \cite[Proposition 3.14 and Corollary 3.17]{Exel:1993pt}; modulo some technicalities, the kernel and the cokernel of $F$ are finitely generated projective $A$-modules, and their difference, that is the `index' of $F$, is the $K_0$ element. Second, the map $E_*$ is then induced by mapping $F:M\to N$ to $F\otimes 1: M\otimes_A E\to N\otimes_A E$, which determines an element of $K_0(B)$. This map is an isomorphism \cite[Theorem 5.3]{Exel:1993pt}. The $K_1$-case is treated via suspensions.

We first deal with the reduced case.  Our module $E$ induces an isomorphism $E_*:K_*(UC^*(X))\to K_*(C^*_u(X))$ as described in the previous paragraph. To prove that $i_{P*}$ is an inverse to this map, it suffices to show that $E_*\circ i_{P*}:K_0(C^*_u(X))\to K_0(C^*_u(X))$ is the identity map. Since $C^*_u(X)$ is unital, any $K_0$-class is represented as $[Q]-[S]$, where $Q,S\in M_n(C^*_u(X))$ are projections. Furthermore, we can assume that $SQ=0$, just by taking $S$ of the form $\diag(0,\mathbb{I}_k)$ (and possibly also enlarging $n$). For the `projective module' description of $K_0$, we can view this class as a difference of two finitely generated projective $C^*_u(X)$-modules $[Q(C^*_u(X)^{\oplus n})]-[S(C^*_u(X)^{\oplus n})]$; or in Exel's description, we may use these two modules and the zero operator between them (cf.\ \cite[below Proposition 3.13]{Exel:1993pt}). Using the proof of \cite[Proposition 3.14]{Exel:1993pt}, $i_{P*}([Q]-[S])$ can be represented as the zero Fredholm operator between $i_P(Q)U^{\oplus n}$ and $i_P(S)U^{\oplus n}$.

Note that  we have an isomorphism of Hilbert modules 
\begin{displaymath}
C^*_u(X)\otimes_{i_P}UC^*(X)\cong PUC^*(X),
\end{displaymath}
defined on simple tensors by $a\otimes b\mapsto i_P(a)b=Pi_P(a)b$. Under this isomorphism, adjointable operators of the form $T\otimes 1$, $T\in C^*_u(X)$, correspond to $Pi_P(T)$. Similarly, we have the matrix version of this: 
$$
P^{\oplus n}UC^*(X)^{\oplus n}\cong (C^*_u(X))^{\oplus n}\otimes_{i_P}UC^*(X),
$$
with $P^{\oplus n}i_P^{(n)}(T)$ corresponding to $T\otimes 1$ for $T\in M_n(C^*_u(X))$.

Via the isomorphism $\pi:UC^*(X)\to \K(E)$ described in the proof of Proposition \ref{prop:morita-eq}, we shall think of $UC^*(X)$ as acting on $E$ on the left by `matrix multiplication, where instead of multiplying the individual entries, we apply the operators to the vectors'.

For notational simplicity, we shall assume that for every $x\in X$, $P_x=P_0\in \B(\mathcal{H}_0)$, a fixed rank-one projection onto a subspace spanned by $\eta_0\in \mathcal{H}_0$ (this case is also all we will actually use). We remark that all the choices are unitarily equivalent on $\ell^2(X,\mathcal{H}_0)$ by a unitary which normalizes $UC^*(X)$.

Note that $C^*_u(X)\otimes_{\pi\circ i_P}E\cong \pi(P)E$, since given a simple tensor $a\otimes b\in C^*_u(X)\otimes_{\pi\circ i_P}E$, we may rewrite it as 
$$
a\otimes b=1\otimes (\pi\circ i_P)(a)b=1\otimes \pi(P)\pi(i_P(a))b
$$ 
(recall that $i_P(1)=P$). We now argue that $\pi(P)E\cong C^*_u(X)$ as Hilbert $C^*_u(X)$-modules.  Given $\xi=(\xi_{xy})\in E$, denote $\lambda(x,y)=\langle \eta_0,\xi_{xy}\rangle\in\C$ and write 
$$
(\pi(P)\xi)_{zw}=\left( \diag(P_0)(\xi_{xy}))\right)_{zw}=P_0\xi_{zw}=\eta_0\langle \eta_0,\xi_{zw}\rangle=\lambda_{zw}\eta_0.
$$ 
The required isomorphism can now be described as $\pi(P)E\ni \pi(P)\xi\mapsto (\lambda(x,y))\in C^*_u(X)$. Consequently, we see that also $C^*_u(Y)^{\oplus n}\otimes_{\pi\circ i_P}E \cong C^*_u(X)^{\oplus n}$ as Hilbert modules.

Under the isomorphism $C^*_u(X)\otimes_{\pi\circ i_P}E\cong C^*_u(X)$, it is easy to compute that the operators of the form $T\otimes 1\in\B(C^*_u(X)\otimes_{\pi\circ i_P}E)$, where $T\in C^*_u(X)$ acts by left multiplication, correspond to $T\in\B(C^*_u(X))\cong C^*_u(X)$. A similar correspondence works for $T\in\B(C^*_u(X)^{\oplus n})\cong M_n(C^*_u(X))$.

Putting the pieces together, for any projection $T\in M_n(C^*_u(X))$, remembering that $P^{\oplus n}=i_P^{(n)}(\mathbb{I}_n)$, we have
\begin{align*}
 \left(i_P^{(n)}(T)U^{\oplus n}\right)\otimes_\pi E &\cong \left((T\otimes1)(C^*_u(X)^{\oplus n}\otimes_{i_P}U)\right)\otimes_\pi E\\
 &\cong (T\otimes 1)(C^*_u(X)^{\oplus n}\otimes_{\pi\circ i_P}E)\\
 &\cong T(C^*_u(X)^{\oplus n}).
\end{align*}
Applying this to $T=Q$ and $T=S$, and noting that $0\otimes 1=0$ for the Fredholm operator, we arrive at the conclusion that the homomorphism $E_*$ is a left inverse to $i_{P*}$, thus finishing the proof in the reduced case.

The proof in the maximal case is essentially the same: the only point we need to argue slightly differently is the isomorphism $C^*_{u,max}(X)\otimes_{\pi\circ i_P} E_{max}\cong C^*_{u,max}(X)$. First, note that $C^*_{u,max}(X)\otimes_{\pi\circ i_P}E_{max}\cong \pi(P)E_{max}$, now thinking of $P\in U\C_0[X]$. Next, for showing that $\pi(P)E_{alg}\cong \C_u[X]$ as right $C^*_{u,max}(X)$-pre-Hilbert-modules we can use the same proof as for the analogous isomorphism in the reduced case. But completing this in the maximal norm yields the desired isomorphism $\pi(P)E_{max}\cong C^*_{u,max}(X)$. The rest of the proof carries over, and we are done.
\end{proof}

\subsection{$K$-theory computations}\label{uniformcase}

In this subsection we complete the proof of Theorem \ref{main} by computing the compositions 
\begin{displaymath}
\alpha_*\circ\beta_*:K_*(C^*_{u,\xx}(X))\to K_*(UC^{*,g}_{\xx}(X));
\end{displaymath}
we show that they are isomorphisms, in a sense inverse to the isomorphisms $E_{*}$ from Section \ref{moritasec}.  As we explain below, this is enough to complete the proof of Theorem \ref{main}. 

Let us introduce some notation. In accordance with \cite[Section 1.3]{Higson:2004la}, define a (`counit') $*$-ho\-mo\-mor\-phism $\eta:\mathcal{S}\to\C$ by $\eta(g)=g(0)$. Let 
$$
\iota:UC^*_{\xx}(X) \to UC^{*,g}_{\xx}(X)
$$
be induced by the canonical inclusion $\mathcal{H}_0\hookrightarrow\mathcal{H}$ as in Lemma \ref{gradelem}.
Let $P\in \K(\mathcal{H}_0)$ be the projection onto $\Span\{e^{-\|v\|^2}\}$, and define a $*$-homomorphism $i_P:\C_u[X]\to U\C[X]$ by $i_P(T)(x,y)=T(x,y)\cdot P$, which extends to give $*$-homomorphisms $i_P:C^*_{u,\xx}(X)\to UC^*_{\xx}(X)$.

\begin{composition}\label{composition}
The composition of asymptotic morphisms 
$$
\alpha_*\circ\beta_*:\mathcal{S}\hat\otimes C^*_{u,\xx}(X) \to UC^{*,g}_{\xx}(X)
$$ 
is homotopic to the $*$-homomorphism 
$$
\iota\circ(\eta\hat\otimes i_P):\mathcal{S}\hat\otimes C^*_u(X)\to UC^{*,g}(X).
$$ 
\end{composition}

We give the proof at the end of this section, but first show how it implies Theorem \ref{main}.

\begin{wrapup}\label{wrapup}
The composition $E_*\circ\phi_* \circ\alpha_*\circ \beta_*: K_*(C^*_{u,\xx}(X))\to K_*(C^*_{u,\xx}(X))$ is the identity map.
\end{wrapup}

\begin{proof}
Using Theorem \ref{composition}, we may replace this composition with $E_*\circ\phi_*\circ\iota_*\circ(\eta\hat\otimes i_P)_*$, which is equal to $E_*\circ(\eta\hat\otimes i_P)_*$ by Lemma \ref{gradelem}, which is the identity by Proposition \ref{kmap}.
\end{proof}

\begin{proof}[Proof of Theorem \ref{main}]
Consider the commutative diagram 
\begin{displaymath}
\vtop{\xymatrix{ K_*(C^*_{u,max}(X)) \ar[d]^{\beta_*} \ar[r]^{\lambda_*}  & K_*(C^*_u(X)) \ar[d]^{\beta_*} \\
K_*(C^*_{u,max}(X,\mathcal{A})) \ar[d]^{\alpha_*} \ar[r]^{\lambda_*}_\cong & K_*(C^*_u(X,\mathcal{A})) \ar[d]^{\alpha_*} \\ K_*(UC^{*,g}_{max}(X)) \ar[d]^{\phi_*} \ar[r]^{\lambda_*} & K_*(UC^{*,g}(X)) \ar[d]^{\phi_*} \\ K_*(UC^{*}_{max}(X)) \ar[d]^{E_*} \ar[r]^{\lambda_*} & K_*(UC^{*}(X)) \ar[d]^{E_*} \\ K_*(C^*_{u,max}(X)) \ar[r]^{\lambda_*}  & K_*(C^*_u(X)) }}.
\end{displaymath}
The two vertical compositions are isomorphisms by Corollary \ref{wrapup}, while the second horizontal arrow is an isomorphism by Proposition \ref{mer}.  The top square thus implies that $\lambda_*:K_*(C^*_{u,max}(X))\to K_*(C^*_{u}(X))$ is injective, while the bottom rectangle (i.e.\ bottom three squares together) implies that it is surjective.
\end{proof}

Theorem \ref{composition} is essentially proved in \cite[Proposition 7.7]{Yu:200ve} for the `usual' (i.e.\ not uniform, or maximal) Roe algebras. We summarize the proof and add some remarks that pertain to the uniform and the uniform--maximal cases.

\begin{proof}[Proof of Theorem \ref{composition}]
The proof has three stages: first we construct a family of maps
$$
\{\gamma(s)_t:\mathcal{S}_0\hat\otimes_{alg}\C_u[X] \to M_2(U\C^g[X])\}_{s\in [0,1], t\in [1,\infty)},
$$
which we show defines a homotopy of asymptotic morphisms
$$\gamma(s):\mathcal{S}\hat\otimes C^*_{u,\xx}(X)\to \mathfrak{A}(M_2(UC^{*,g}_\xx(X)));$$
second, we use Lemma \ref{as-compose} to show that the asymptotic family $\gamma_t:=\gamma(1)_t$ represents the composition of $\alpha$ and $\beta$ in $E$-theory; thirdly, we construct a homotopy between the asymptotic morphism $\gamma':=\gamma(0)$ and that defined by the $*$-homomorphism $\iota\circ(\eta\hat\otimes i_P)$.

A general remark is perhaps in order.  The asymptotic morphisms (i.e.\ variations on $\alpha$, $\beta$, $\gamma$) that we are using are all defined `entrywise' for $X$-by-$X$ matrices of finite propagation.  Now, to show, for example, that two such asymptotic morphisms are asymptotically equivalent, either in the reduced or maximal cases, it suffices to prove the requisite estimates entrywise, \emph{as long as} they are uniform across entourages of the form $\{(x,y)\in X\times X~|~d(x,y)\leq R\}$\footnote{This statement uses bounded geometry, which implies that there is a uniform bound on the number of non-zero entries in the rows or columns of a finite propagation matrix} .  One can then appeal to Lemma \ref{fundlem} to prove (for example, again) asymptotic equivalence over the entire $C^*$-algebra (with respect to either the reduced or maximal completion).  Our arguments below essentially proceed entrywise; we make use of the remark above without further comment. \\

We now define $\gamma(s)_t$.  

Let then $U_{x,t}$ be the unitary operator on $\mathcal{H}$ induced by the translation $v\mapsto v-tf(x)$ on $V$. Let $R(s)=\left(\begin{smallmatrix}\cos(\pi s/2)& \sin(\pi s/2)\\ -\sin(\pi s/2) & \cos(\pi s/2)\end{smallmatrix}\right)$, $s\in [0,1]$ be the rotation matrices. Denote
\begin{displaymath}
U_{x,t}(s)=R(s)\begin{pmatrix}U_{x,t}&0\\0&1\end{pmatrix}R(s)^{-1}\in\B(\mathcal{H}\oplus\mathcal{H})
\end{displaymath}
and 
\begin{equation}\label{eq:gammadef0}
(\gamma_t(g\hat\otimes T))(x,y)=\theta_t^N(x)(\beta_{f(x)\in W_N(x)}(g_{t}))\cdot T(x,y)
\end{equation}
where $N$ is large enough, depending on the propagation of $T$. Let us note that large $N$ is needed in order to be able to `switch from $x$ to $y$' when proving that each $\gamma_t$ defines an asymptotic family (an outline of the argument for this is given below). 
We define our family of maps $\gamma(s)_t:\mathcal{S}_0\hat\otimes \C_u[X]\rightsquigarrow M_2(U\C^g[X])$ by letting
\begin{align*}
\gamma(s)_t(g\hat\otimes T)(x,y)&=U_{x,t}(s)\begin{pmatrix}(\gamma_t(g\hat\otimes T))(x,y)&0\\0&0\end{pmatrix}U_{x,t}(s)^{-1}\\
 &=U_{x,t}(s)\begin{pmatrix}\theta_t^N(x)(\beta_{f(x)\in W_N(x)}(g_t))&0\\0&0\end{pmatrix}U_{x,t}(s)^{-1}\cdot T(x,y).
\end{align*}

The basic point of the homotopy in the $s$ variable is that the Bott maps $\beta(x)$ we use include a point in $V$ as $f(x)$, and not as $0$; moreover, this is reflected in the $B_{N,t}$ operators, since the Clifford multiplication operators `$C$' incorporate this in their definition. The homotopy $\gamma(s)$ of asymptotic morphisms interpolates between the two inclusions: $\gamma(1)$ includes a point as $f(x)$, and is thus closely related to $\alpha\circ\beta$; $\gamma(0)$ includes all points as $0\in V$, and is thus more closely related to  $\iota\circ(\eta\hat\otimes i_P)$.

We next argue that for every $s\in[0,1]$, $\gamma(s)_t$ is an asymptotic family; this has essentially been done in the proof of Claim \ref{ashomo}.  Indeed, multiplying out the matrices in the formula above for $\gamma(s)_t$ shows that it suffices to prove that the formula in line (\ref{eq:gammadef0}) defines an asymptotic morphism.  Consider now a subset of $\mathcal{S}_0\hat\otimes_{alg}\C_u[X]$ satisfying the following.  `All elements are a sum of finitely many elementary tensors $f\hat{\otimes}T$ from $\mathcal{S}_0$ and $\C_u[X]$ such that:
\begin{itemize}
\item there exists $R>0$ so that $f$ is supported in $[-R,R]$;
\item there exists $c>0$ so that $\|df/dx\|_{C_0(\R)}\leq c$;
\item there exists $S>0$ so that $T$ is of propagation at most $S$.' 
\end{itemize}
One checks that on such a subset the family $\{\beta_t\}$ satisfies the estimates needed to show that it is an asymptotic family.  Moreover, each of the maps $\beta_t$ take such a subset into a subset of $\C[X,\mathcal{A}]$ where the conditions from Definition \ref{tra} are satisfied \emph{uniformly}, and on such a subset of $\C[X,\mathcal{A}]$, the family $\alpha_t$ satisfies the estimates needed to show that \emph{it} is an asymptotic family uniformly.  This implies that $\gamma_t$, whence also $\gamma(s)_t$, defines an asymptotic family\footnote{This sort of argument can be used to compute the composition of $\alpha$ and $\beta$ along the lines of \cite{Connes:1990kx} and without using Lemma \ref{as-compose}; we prefer our set-up, however, as it seems more general and elegant.}.  

A similar argument shows that the range of each $\gamma(s)_t$ really is in $U\C^g[X]$.  Indeed, it is clear that the rank of the finite-rank approximants to each entry of the two-by-two matrix defining $\gamma(s)_t(a)$ is just the same as for the operator $\gamma_t(a)$ itself, and we can study this using the discussion from Section \ref{bdsec}.

A remark about norms is in order. We can extend $\gamma(s)$ to the $red\to red$ situation by an argument similar to the one used for $\beta$ in the proof of Claim \ref{completions}: we estimate $\|\gamma(s)(g\hat\otimes T)\|\leq \|g\|\|T\|$ analogously to the method used for $\beta_t$ in \cite[Proof of Lemma 7.6]{Yu:200ve}, then extend $\gamma(s)$ to the algebraic tensor product of $\mathcal{S}$ with $C^*_u(X)$, and finally extend it to the maximal (and hence reduced) tensor product. In the $max\to max$ case, we argue exactly as for $\beta$ in Claim \ref{completions}.

We have argued that each $\gamma(s)$ is an asymptotic morphism; we need to show that it defines a homotopy of asymptotic morphisms.  It is clear from the formula, however, that for a fixed $a\in \mathcal{S}_0\hat\otimes_{alg} \C_u[X]$ and $t\in[1,\infty)$, $[0,1]\ni s\mapsto \gamma(s)_t(a)$ is a norm-continuous path in $U\C^g[X]$, however; it follows from this that $\{\gamma(s)\}_{s\in [0,1]}$ is a genuine homotopy of asymptotic morphisms. Note in particular, then, that $\gamma(1)$, $\gamma(0)$ represent the same element in $E$-theory.\\

We now show that $\gamma(1)$ represents the composition of $\alpha$ and $\beta$.

To show that we can use the `naive' composition $\alpha_t\circ\beta_t$ as a representative of the composition of the asymptotic morphisms $\alpha$ and $\beta$, we use Lemma \ref{as-compose}, proved in Appendix \ref{appendix1}. For $r\in(0,1]$, let
\begin{equation}\label{eq:gammadef}
(\gamma^{(r)}_t(g\hat\otimes T))(x,y)=\theta_t^N(x)(\beta_{f(x)\in W_N(x)}(g_{rt}))\cdot T(x,y).
\end{equation}
As before, the choice of $N$ does not asymptotically matter. An argument completely analogous to the one given above for the endpoint $\gamma(1)$ of the homotopy $\gamma(s)$ (replacing $t$ with $rt$ as necessary) shows that $\gamma^{(r)}$, a priori $\mathcal{S}_0\hat\otimes \C_u[X]\rightsquigarrow U\C^g[X]$ extends to an $r$-parametrised family of asymptotic morphisms in both $red\to red$ and $max\to max$ cases. It is clear from the definition of $\gamma^{(r)}$ that it is in fact asymptotic to the composition $\alpha_t\circ\beta_{rt}$; thus Lemma \ref{as-compose} applies, and we can represent the composition of asymptotic morphisms $\alpha\circ\beta$ by the asymptotic morphism $\gamma^{(1)}$.\\

Finally, we prove that $\gamma(0)$ is homotopic to $\iota\circ(\eta\hat\otimes i_P)$.

Note that $\gamma(0)_t$ is asymptotic to the morphism $\left(\begin{smallmatrix}\gamma'_t&0\\0&0\end{smallmatrix}\right)$, where $\gamma'$ is defined by
\begin{equation*}
(\gamma'_t(g\hat\otimes T))(x,y)=\theta_t^N(x)\circ \beta_{0\in W_N(x)}(g_t)\cdot T(x,y),
\end{equation*}
where $N$ is sufficiently large. Using the proof \cite[(unnumbered) Proposition in Appendix B; `Mehler's formula']{Higson:1999be} and Section 5 in the same paper, the family $g\mapsto \theta_t^N(x)(\beta_{0\in W_N(x)})(g_t)$ is asymptotic to the family of $*$-homomorphisms $g\mapsto g_{t^2}(B_{t}(x))$, where
$$
B_t(x):=t_0(D_0+C_0')+t_1(D_0+C_1)+\dots+t_n(D_n+C_n)+\dots,\quad\text{ where }t_j=1+j/t,
$$
and where $C_0'$ is the Clifford multiplication operator on $V_0(x)=W_1(x)$, now induced by the function $v\mapsto v$, instead of $v\mapsto v-f(x)$. Furthermore, this family is homotopic to the $*$-homomorphism $g\mapsto g(0)P$, where $P$ is the rank one projection onto the kernel of $B_t(x)$; this no longer depends on $x$, and is indeed equal to $\Span\{e^{-\|v\|^2}\}$ (see for example \cite[page 30]{Higson:1999be}). Finally, one observes that all this happens uniformly in $x$, whence we obtain a homotopy of asymptotic morphisms from $\gamma'$ to the $*$-homomorphism $\mathcal{S}\hat\otimes C^*_{u,\xx}(X)\to UC^{*,g}_{\xx}(X)$, given on simple tensors by the formula $g\hat\otimes T\mapsto \eta(g)i_P(T)$.
\end{proof}

\begin{appendix}

\section{Appendix: $K-$ and $E-$ theory conventions, and a lemma about composing asymptotic morphisms}
\label{appendix1}

As mentioned at the start of Section \ref{mainproof}, we use graded $K$-theory, and graded $E$-theory for some of our main computations.  As there are several possible descriptions of these theories, not all of which agree (or even make sense) for non-separable $C^*$-algebras, it seemed worthwhile to summarise our conventions here.  Our main reference for graded $K$-theory and $E$-theory is \cite{Higson:2004la}, which in turn refers to the ungraded case covered in \cite{Guentner:2000fj} for many proofs.  Note, however, that our definition of graded $E$-theory does \emph{not} match that given in \cite[Definition 2.1]{Higson:2004la}\footnote{In our notation, \cite[Definition 2.1]{Higson:2004la} defines $E(A,B):=\llbracket \mathcal{S}\hat{\otimes}A\hat{\otimes}\mathcal{K}(\mathcal{H}) , B\hat{\otimes}\mathcal{K}(\mathcal{H})\rrbracket_1$}: indeed, \cite[Theorem 2.16]{Guentner:2000fj} implies that \cite[Definition 2.1]{Higson:2004la} and the definition in line (\ref{egroups}) below are equivalent in the separable case, but this is not at all clear in the non-separable case.  It seems that the definition used below has better properties in the non-separable case.  

To avoid unnecessary multiplication of adjectives, throughout this appendix `$C^*$-algebra' means `graded $C^*$-algebra', and `$*$-homomorphism' means `graded $*$-homomorphism'.\\

The authors of \cite{Higson:2004la} define functors $\mathfrak{T}$, $\mathfrak{T}_0$ and $\mathfrak{A}$ from the category of $C^*$-algebras and $*$-ho\-mo\-mor\-phisms into itself. On the objects, $\mathfrak{T}$ takes a $C^*$-algebra $B$ to the $C^*$-algebra of continuous, bounded functions from $[1,\infty)$ into $B$, while $\mathfrak{T}_0B$ is the ideal of those functions in $\mathfrak{T}B$ which vanish at infinity.  Finally, $\mathfrak{A}B=\mathfrak{T}B/\mathfrak{T}_0B$. 

An \emph{asymptotic morphism} $\varphi$ from a $C^*$-algebra $A$ into a $C^*$-algebra $B$ is a $*$-homomorphism from $A$ into $\mathfrak{A}B$. This is essentially equivalent to giving a family of maps $\varphi_t:A\to B$, $t\in [1,\infty)$ that satisfy certain  conditions: see \cite[Definition 1.18]{Higson:2004la}.   We say that a family of maps $\{\varphi_t:A\to B\}_{t\in[1,\infty)}$ is an \emph{asymptotic family} if it satisfies the conditions from \cite[Definition 1.18]{Higson:2004la}.  Throughout this piece we use the notation `$\varphi_t:A\leadsto B$' to mean `$\{\varphi_t\}_{t\in[1,\infty)}$ is a $[1,\infty)$-parametrised family of maps from $A$ to $B$'. 

To properly define composition of asymptotic morphisms without assumptions on separability of the $C^*$-algebras involved, one needs to consider the functors $\mathfrak{A}^n$, the composition of the functor $\mathfrak{A}$ with itself $n$ times, and an appropriate notion of homotopy: two $*$-homomorphisms $\varphi_0,\varphi_1:A\to\mathfrak{A}^nB$ are \emph{$n$-homotopic} if there is a closed interval $I$ and a $*$-homomorphism $\varphi:A\to \mathfrak{A}^nC(I,B)$ from which $\varphi_0$ and $\varphi_1$ can be recovered upon composing with the evaluations at the endpoints of $I$. One then defines $\llbracket A,B\rrbracket_n$ to be the set of $n$-homotopy classes of $*$-homomorphisms from $A$ to $\mathfrak{A}^nB$. 

We denote by $\alpha_B:B\to\mathfrak{A}B$ the $*$-homomorphism which maps $b\in B$ to the class of the constant function $[1,\infty)\ni t\mapsto b\in B$. Composition with $\alpha_{\mathfrak{A}^nB}$ induces a map $\llbracket A,B\rrbracket_n\to \llbracket A,B\rrbracket_{n+1}$. Under these maps, $(\llbracket A,B\rrbracket_n)_{n\in \N}$ forms a directed system; its direct limit is denoted by $\llbracket A,B\rrbracket_\infty$. These morphism sets can be made into a category by defining the composition of two $*$-homomorphisms $\varphi:A\to \mathfrak{A}^jB$ and $\psi:B\to \mathfrak{A}^kC$ to be the element of $\llbracket A,B\rrbracket_\infty$ represented by
$$
A\stackrel{\scriptscriptstyle\varphi}{\rightarrow}\mathfrak{A}^jB 
\stackrel{\scriptscriptstyle{\mathfrak{A}^j(\psi)}}{\longrightarrow}\mathfrak{A}^{j+k}C.
$$

Finally, if $A$, $B$ are $C^*$-algebras and $\mathcal{H}$ a fixed graded separable infinite-dimensional Hilbert space, the $E$-theory group $E(A,B)$ is defined to be
\begin{equation}\label{egroups}
E(A,B):=\llbracket \mathcal{S}\hat{\otimes}A\hat{\otimes}\mathcal{K}(\mathcal{H}) , B\hat{\otimes}\mathcal{K}(\mathcal{H})\rrbracket_\infty.
\end{equation}
Note that asymptotic morphisms $\mathcal{S}\hat{\otimes} A \leadsto B$, or simply $A\leadsto B$ induce elements of $E(A,B)$ by tensoring with the identity on $\mathcal{K}(\mathcal{H})$ and using the counit $\eta:\mathcal{S}\to \C$ as appropriate.
Composition in $E$-theory is then defined using the coproduct $\Delta$ for $\mathcal{S}$: the composition of $\varphi\in E(A,B)$ and $\psi\in E(B,C)$ is defined to be 
$$
 \mathcal{S}\hat{\otimes}A\hat{\otimes}\mathcal{K}(\mathcal{H}) \stackrel{\scriptstyle\Delta}{\rightarrow} \mathcal{S}\hat{\otimes}\mathcal{S}\hat{\otimes}A\hat{\otimes}\mathcal{K}(\mathcal{H}) \stackrel{\scriptstyle\varphi}{\rightarrow} \mathfrak{A}^j( \mathcal{S}\hat{\otimes}B\hat{\otimes}\mathcal{K}(\mathcal{H}) )  \stackrel{\scriptstyle\mathfrak{A}^j\psi}{\rightarrow} \mathfrak{A}^{k+j}(C\hat{\otimes}\mathcal{K}(\mathcal{H})).
$$

Having set up all of these preliminaries, the graded $K$-theory of a graded $C^*$-algebra $A$ can be defined to be $K_0(A):=E(\C,A)$, with higher $K$-groups defined via suspension $K_{n}(A):=E(\C,C_0(\R^n)\hat{\otimes} A)$.  Bott periodicity holds, so there are essentially only two higher $K$-groups; we write $K_*(A)$ for the graded abelian group $K_0(A)\oplus K_1(A)$.  It is then immediate from the above definitions that an element $\varphi\in E(\C,A)$ induces a map $\varphi_*:K_*(A)\to K_*(B)$.  Moreover, \cite[Proposition 1.10 and Lemma 2.2]{Higson:2004la} and \cite[Proposition 2.9]{Guentner:2000fj} imply that if $A$ is trivially graded, then this definition of $K$-theory agrees with any of the usual ones (this fact does not require separability).

\begin{asnot}\label{asnot}
For the reader's convenience, we summarise our notational conventions.
\begin{itemize}
\item `$\alpha_t:A\leadsto B$' denotes a family of maps between $*$-algebras $A$ and $B$ parametrised by $[1,\infty)$; we say that $\{\alpha_t\}_{t\in [1,\infty)}$ is an \emph{asymptotic family} if the $\alpha_t$ satisfies the conditions from \cite[Definition 1.18]{Higson:2004la}. 
\item An \emph{asymptotic morphism} is a $*$-homomorphism $\alpha:A\to \mathfrak{A}B$; in the main body of the paper, such are usually induced by asymptotic families as above.  We write $\alpha$ for the asymptotic morphism induced by an asymptotic family $\alpha_t:A\leadsto B$; we do not distinguish between an asymptotic morphism $\alpha:A\to\mathfrak{A}B$ (or $\alpha:\mathcal{S}\hat\otimes A \to \mathfrak{A}B$) and the element of $E$-theory $\alpha\in E(A,B)$ that it defines.  
\item We write $\alpha_*:K_*(A)\to K_*(B)$ for the homomorphism (of graded abelian groups) induced by $\alpha\in E(A,B)$ (note, of course, that the original source of such an $\alpha$ is usually an asymptotic family $\alpha_t:A\leadsto B$, or $\alpha_t:\mathcal{S}\hat\otimes A \leadsto B$).
\end{itemize}
\end{asnot}

\begin{as-compose}\label{as-compose}
Assume that we are given two asymptotic morphisms $\varphi_t:A\leadsto B$ and $\psi_t:B\leadsto C$. Moreover, assume that for every $s\in (0,1]$, the composition $\psi_t\circ\varphi_{st}:A\leadsto C$ \emph{is} an asymptotic morphism. Then the composition $\psi\circ \varphi$ in $E(A,C)$ is represented by the asymptotic family $\psi_t\circ \varphi_t:A\leadsto C$.
\end{as-compose}

\begin{acrem}
Simple examples show that even if $\varphi_t:A\leadsto B$, $\psi_t:B\leadsto C$ and the `naive composition' $\psi_t\circ\varphi_t:A\leadsto C$ all happen to be asymptotic morphisms, then one need not have that $\psi_*\circ\phi_*:K_*(A)\to K_*(C)$ is the same as the map on $K$-theory induced by $\psi_t\circ\varphi_t$ (and so in particular, the composition of $\psi$ and $\varphi$ in $E$-theory is not equal to the $E$ theory class induced by $\{\psi_t\circ\varphi_t\}$).  Indeed, let $A=B=C=C_0(\R)$.  Let $\{u_t\}_{t\in [1,\infty)}$ be a continuous approximate unit for $C_0(\R)$ such that each $u_t$ is supported in $[-t,t]$, and let $h_t:(2t, 3t)\to \R$ be a continuous family of orientation preserving homeomorphisms.  Let $\psi_t:C_0(\R) \to C_0(\R)$ be defined by $f\mapsto u_t\cdot f$.  Let $\iota_t:C_0(2t, 3t)\to C_0(\R)$ be the usual $*$-homomorphism induced by an open inclusion and let $\varphi_t:C_0(\R)\to C_0(\R)$ be given by $f\mapsto \iota_t(f\circ h_t)$.  Then both $\psi_t$, $\varphi_t$ are asymptotic morphisms that induce the identity on $K$-theory.  The naive composition $\psi_t\circ\varphi_t$ is the zero map, so also an asymptotic morphism, but certainly does not induce the identity on $K$-theory.  
\end{acrem}

\begin{proof}
The proof is inspired by the proof of \cite[Lemma 2.17]{Guentner:2000fj}. We have the following diagram:
\begin{equation}\label{eq:2-htpy}
\vcenter{\xymatrix{
A\ar[r]^{\varphi}\ar[d]_{\psi_t\circ\varphi_t} & \mathfrak{A}B\ar[d]^{\mathfrak{A}(\psi)} \\
\mathfrak{A}C\ar[r]^{\alpha_{\mathfrak{A}C}} & \mathfrak{A}^2C,
}}
\end{equation}
where the composition across the upper-right corner is the `correct' definition of the composition of the asymptotic morphisms; and the other composition is what we would like the product to be represented by. All arrows are $*$-homomorphisms (by assumption). We need to prove that the diagram commutes up to 2-homotopy.

Recall from \cite[Remark 2.11]{Guentner:2000fj} that we can think of elements of $\mathfrak{A}^2IC$ as represented by functions in $\mathfrak{T}^2IC$, i.e.\ continuous bounded functions on three variables, $t_1,t_2\in [1,\infty)$ and $s\in[0,1]$ into $C$ (satisfying certain extra continuity conditions). The $C^*$-algebra $\mathfrak{A}^2IC$ is actually the quotient of $\mathfrak{T}^2IC$ corresponding to the $C^*$-seminorm
$$
\|F\|_{\mathfrak{A}^2}=\limsup_{t_1\to\infty}\limsup_{t_2\to\infty}\sup_s\|F(t_1,t_2,s)\|.
$$
Similarly for $\mathfrak{A}^2C$, we just don't use the variable $s$. Using this description, the composition $\mathfrak{A}(\psi)\circ\varphi$ from \eqref{eq:2-htpy} can be expressed as the map assigning to every $a\in A$ the function $\psi_{t_2}(\varphi_{t_1}(a))$, and the other composition from \eqref{eq:2-htpy} as the function $\psi_{t_2}(\varphi_{t_2}(a))$.

We now construct the required 2-homotopy $H:A\to \mathfrak{T}^2IC$:
$$
a\mapsto H(t_1,t_2,s)=\begin{cases}
\psi_{t_2}(\varphi_{t_1}(a)) & \text{ if }t_1>st_2\\
\psi_{t_2}(\varphi_{st_2}(a)) & \text{ if }t_1\leq st_2.
\end{cases}
$$
This finishes the proof.
\end{proof}

\section{Appendix: List of notation}
\label{appendix2}

We include the following list of notation for the reader's convenience, partly as so much of our notation is imported from \cite{Yu:200ve}.  Definitions for most of the objects below are included in the main body of the text.
\def\arraystretch{1.2}
\def\tabcolsep{.5mm}
\def\cg#1{\cite[#1]{Yu:200ve}}
\begin{longtable}{lcp{.8\textwidth}r}
$H$ &$=$&infinite-dimensional separable Hilbert space. & \\
$f$ &$:$& $X\to H$, the fixed coarse embedding of $X$ into $H$. & \\
$W_n(x)$ &$=$& $\Span\{f(y)\in H\mid d(x,y)\leq n^2\}$. & \cg{p.212} \\
$V_n(x)$ &$=$& $W_{n+1}(x)\ominus W_n(x)$, $V_0(x)=W_1(x)$. & \cg{p.228} \\
$W(x)$ &$=$& $\bigcup_{n\in\N}W_n(x)~~ (=V)$. & \cg{p.211} \\
$V_a,V_b$ &$=$& finite-dimensional affine subspaces of $V$. & \cg{p.211} \\
$V_a^0$ &$=$& the linear subspace of differences of elements of $V_a$. & \cg{p.211}   \\
$\Cliff_\C(V_a^0)$ &$=$& the complex Clifford algebra of $V_a^0$. & \\
$\mathcal{C}(V_a)$ &$=$& $C_0(V_a,\Cliff_\C(V_a^0))$. & \cg{p.211} \\
$\mathcal{S}$ &$=$& $C_0(\R)$, graded by even and odd functions. & \\
$\eta$ &$:$& $\mathcal{S}\to \C$, $g\mapsto g(0)$, the `counit' $*$-homomorphism. & \\
$\mathcal{A}(V_a)$ &$=$& $\mathcal{S}\hat\otimes\mathcal{C}(V_a)$. & \cg{p.211} \\
\multicolumn{3}{l}{Assume now that $V_a\subseteq V_b$.} & \\
$V_{ba}^0$ &$=$& the orthogonal complement of $V_a^0$ in $V_b^0$. & \cg{p.211} \\
$\tilde h$ &$=$& the extension of $h\in\mathcal{C}(V_a)$ to a multiplier of $\mathcal{C}(V_b)$ by the formula $\tilde h(v_b)=h(v_a)$, where $v_b=v_{ba}+v_a\in V_{ba}^0+V_a=V_b$. & \cg{p.211} \\
$X$ &$:$& $\R\to\R$, $t\mapsto t$, thought of as an unbounded multiplier of $\mathcal{S}$. & \cg{p.211} \\
$C_{ba}$ &$:$& $V_{b}\to \Cliff_\C(V_{b}^0)$, $v_b\mapsto v_{ba}\in V_{ba}^0\subset \Cliff_\C(V_b^0)$, thought of as an unbounded multiplier of $\mathcal{C}(V_b)$. & \cg{p.211} \\
$\beta_{ba}$ &$:$& $\mathcal{A}(V_a)\to\mathcal{A}(V_b)$, $g\hat\otimes h\mapsto g(X\hat\otimes 1+1\hat\otimes C_{ba})(1\hat\otimes\tilde h)$, a $*$-homomorphism. & \cg{p.211}\\
$\mathcal{A}(V)$ &$=$& $\lim_{V_a\subset V}\mathcal{A}(V_a)$, using $\beta_{ba}$s as connecting maps. & \cg{p.212} \\
$\beta_n(x)$ &$:$& $\mathcal{A}(W_n(x))\to\mathcal{A}(V)$. & \cg{p.213} \\
$\mathcal{H}_a$ &$=$& $L^2$-functions from $V_a$ into $\Cliff_\C(V_a^0)$, carries a representation of $\mathcal{A}(V_a)$.  & \cg{p.227} \\
$T_{ba}$ &$:$& $\mathcal{H}_a\to\mathcal{H}_b$, a unitary, via $\mathcal{H}_a\ni \xi\mapsto \xi_0\hat\otimes\xi\in\mathcal{H}_{ba}\hat\otimes\mathcal{H}_a$, where $\xi_0(v_{ba})=\pi^{-\dim(V_{ba})/4}\exp(-\|v_{ba}\|^2/2)$. & \\
$\mathcal{H}$ &$=$& $\lim_{V_a\subset V}\mathcal{H}_a$, using $T_{ba}$s; carries a representation of $\mathcal{A}(V)$. & \cg{p.227} \\
$s_a$ &$=$& the Schwartz subspace of $\mathcal{H}_a$. & \cg{p.227}\\
$s$ &$=$& $\lim_{V_a\subset V}s_a$ (algebraic limit), the Schwartz subspace of $\mathcal{H}$. & \cg{p.227}\\
$D_a$ &$:$& $s\to s$ is the Dirac operator `on $V_a$', $D_a\xi=\sum(-1)^{\deg\xi}\frac{\partial\xi}{\partial x_i}v_i$. & \cg{p.227} \\
$C_a$ &$:$& $s\to s$ is the Clifford multiplication operator `on $V_a$', $(C_a\xi)(v_b)=v_a\xi(v_b)$, where $V_a\subset V_b$, $\xi\in s_b$, $v_b\in V_b$, $v_b=v_{ba}+v_a$. & \cg{p.227} \\
$B_{n,t}(x)$ &$=$& $t_0D_0+\dots+t_{n-1}D_{n-1}+t_n(D_n+C_n)+t_{n+1}(D_{n+1}+C_{n+1})+\dots$, where $t_j=1+j/t$, $D_n$ and $C_n$ are the Dirac and Clifford operators on $V_n(x)$.  & \ \cg{p.228} \\
$h_t(x)$ &$=$& $h(t^{-1}x)$ for a function $h$. & \cg{p.229} \\
$\pi(h)\xi$ &$=$& $\tilde h\xi$ is the multiplication action: $\mathcal{C}(W_n(x))\to\B(\mathcal{H}_b)\hookrightarrow \B(\mathcal{H})$, where $\xi\in\mathcal{H}_b$, $V_a=W_n(x)\subset V_b$. & \cg{p.229} \\
$\theta_t^n(x)$ &$:$& $\mathcal{A}(W_n(x))\rightsquigarrow\K(\mathcal{H})$, $\theta_t^n(x)(g\hat\otimes h)=g_t(B_{n,t}(x)|_{W_n(x)})\pi(h_t)$, $g\in\mathcal{S}$, the finite-dimensional Dirac morphism.& \cg{p.229}\\
$\alpha_t$ &$:$& $C^*_{u,\xx}(X,\mathcal{A})\rightsquigarrow UC^*_{\xx}(X)$, the Dirac morphism (more in Section \ref{mainproof} above). & \\
$\beta_{f(x)\in W_N(x)}$ &$:$& $\mathcal{A}(\{f(x)\})=\mathcal{S}\to\mathcal{A}(W_N(x))$, the $*$-homomorphism (`Bott morphism') associated to the inclusion of the zero-dimensional affine space $\{f(x)\}$ into $W_N(x)$. & \\
$\beta(x)$ &$:$& $\mathcal{S}\to \mathcal{A}(V)$, the Bott morphism for $\{f(x)\}\subset V$. & \cg{p.235}\\
$\beta_t$ &$:$& $\mathcal{S}\hat\otimes C^*_{u,\xx}(X)\rightsquigarrow C^*_{u,\xx}(X,\mathcal{A})$, the Bott morphism (more in Section \ref{mainproof} above). & \\
$\iota$ &$:$& $UC^*_\xx(X)\to UC^{*,g}_\xx(X)$, `the' inclusion of the ungraded uniform algebra in the graded uniform algebra. &  \\
$\phi$ &$:$& $UC^{*,g}_\xx(X)\to UC^*_\xx(X)\hat\otimes\Cliff_\C(\R^2)$, `the' isomorphism of the graded uniform algebra with the `stabilised' ungraded version. &   

\end{longtable}

\end{appendix}


\begin{thebibliography}{10}
\providecommand{\url}[1]{\texttt{#1}}
\providecommand{\urlprefix}{URL }
\expandafter\ifx\csname urlstyle\endcsname\relax
  \providecommand{\doi}[1]{doi:\discretionary{}{}{}#1}\else
  \providecommand{\doi}{doi:\discretionary{}{}{}\begingroup
  \urlstyle{rm}\Url}\fi
\providecommand{\eprint}[2][]{\url{#2}}

\bibitem{Anantharaman-Delaroche:2000mw}
\emph{C.~Anantharaman-Delaroche and J.~Renault}, Amenable groupoids,
  L'enseignement Math\'{e}matique (2000).

\bibitem{Brodzki:2007mi}
\emph{J.~Brodzki, G.~Niblo, and N.~Wright}, Property {A}, partial translation
  structures and uniform embeddings in groups, Journal of the London
  Mathematical Society, \textbf{76}(2) (2007), 479--497.

\bibitem{Brown:2008qy}
\emph{N.~Brown and N.~Ozawa}, ${C}^*$-Algebras and Finite-Dimensional
  Approximations, vol.~88 of \emph{Graduate Studies in Mathematics}, American
  Mathematical Society (2008).

\bibitem{Connes:1994zh}
\emph{A.~Connes}, Noncommutative Geometry, Academic Press (1994).

\bibitem{Connes:1990kx}
\emph{A.~Connes and N.~Higson}, Almost homomorphisms and ${KK}$-theory,
  Preprint (available at www.math.psu.edu/higson) (1990).

\bibitem{Exel:1993pt}
\emph{R.~Exel}, A {F}redholm operator approach to {M}orita equivalence,
  ${K}$-theory, \textbf{7}(3) (1993), 285--308.

\bibitem{Gong:2008ja}
\emph{G.~Gong, Q.~Wang, and G.~Yu}, Geometrization of the strong {N}ovikov
  conjecture for residually finite groups, Journal f\"{u}r die reine und
  angewandte Mathematik, \textbf{621} (2008), 159--189.

\bibitem{Guentner:2000fj}
\emph{E.~Guentner, N.~Higson, and J.~Trout}, Equivariant {E}-theory, Memoirs of
  the American Mathematical Society, \textbf{148}(703) (2000).

\bibitem{Higson:1999km}
\emph{N.~Higson}, Counterexamples to the coarse {B}aum-{C}onnes conjecture
  (1999), available on the author's website.

\bibitem{Higson:2004la}
\emph{N.~Higson and E.~Guentner}, Group ${C}^*$-algebras and ${K}$-theory, in
  Noncommutative Geometry, no. 1831 in Springer Lecture Notes, Springer (2004),
  137--252.

\bibitem{Higson:2001eb}
\emph{N.~Higson and G.~Kasparov}, ${E}$-theory and ${KK}$-theory for groups
  which act properly and isometrically on {H}ilbert space, Inventiones
  Mathematicae, \textbf{144} (2001), 23--74.

\bibitem{Higson:1999be}
\emph{N.~Higson, G.~Kasparov, and J.~Trout}, A {B}ott periodicity theorem for
  infinite dimensional {H}ilbert space, Advances in Mathematics, \textbf{135}
  (1999), 1--40.

\bibitem{Hulanicki:1967lt}
\emph{A.~Hulanicki}, Means and {F}\o lner conditions on locally compact groups,
  Studia Mathematica, \textbf{27} (1966), 87--104.

\bibitem{Lance:1995ys}
\emph{E.~C. Lance}, Hilbert $C^*$-modules (a toolkit for operator algebraists),
  Cambridge University Press (1995).

\bibitem{Lubotzky:1994tw}
\emph{A.~Lubotzky}, Discrete groups, Expanding Graphs and Invariant Measures,
  Birkh\"{a}user (1994).

\bibitem{Lubotzky:1997wq}
\emph{A.~Lubotzky}, Eigenvalues of the {L}aplacian, the first {B}etti number
  and the congruence subgroup problem, The Annals of Mathematics, \textbf{145}
  (1997), 441--452.

\bibitem{Lubotzky:2004xw}
\emph{A.~Lubotzky and Y.~Shalom}, Finite representations in the unitary dual
  and {R}amanujan groups, in Discrete geometric analysis, no. 347 in
  Contemporary Mathematics, American Mathematical Societyc (2004), 173--189.

\bibitem{Oyono-Oyono:2009ua}
\emph{H.~Oyono-Oyono and G.~Yu}, {K}-theory for the maximal {R}oe algebra of
  certain expanders (2009), accepted to Journal of Functional Analysis.

\bibitem{Pedersen:1979zr}
\emph{G.~K. Pedersen}, $C^*$-Algebras and their Automorphism Groups, Academic
  Press (1979).

\bibitem{Roe:1988qy}
\emph{J.~Roe}, An index theorem on open mainfolds, {I}, Journal of differential
  geometry, \textbf{27} (1988), 87--113.

\bibitem{Roe:1993lq}
\emph{J.~Roe}, Coarse cohomology and index theory on complete {R}iemannian
  manifolds, Memoirs of the American Mathematical Society, \textbf{104}(497)
  (1993).

\bibitem{Roe:2003rw}
\emph{J.~Roe}, Lectures on Coarse Geometry, vol.~31 of \emph{University Lecture
  Series}, American Mathematical Society (2003).

\bibitem{Serre:1970df}
\emph{J.-P. Serre}, Le probleme des groupes de congruence pour
  $\mathbf{{SL}}_2$, The Annals of Mathematics, \textbf{92}(3) (1970),
  489--527.

\bibitem{Skandalis:2002ng}
\emph{G.~Skandalis, J.-L. Tu, and G.~Yu}, The coarse {B}aum-{C}onnes conjecture
  and groupoids, Topology, \textbf{41} (2002), 807--834.

\bibitem{Tu:1999qm}
\emph{J.-L. Tu}, The {B}aum-{C}onnes conjecture for groupoids, in
  \emph{J.~Cuntz and S.~Echterhoff} (eds.), ${C^*}$-Algebras (1999), 227--242.

\bibitem{Tu:1999bq}
\emph{J.-L. Tu}, La conjecture de {B}aum-{C}onnes pour les feuilletages
  moyennables, ${K}$-theory, \textbf{17} (1999), 215--264.

\bibitem{Tu:2001bs}
\emph{J.-L. Tu}, Remarks on {Y}u's property {A} for discrete metric spaces and
  groups, Bulletin de la Soci\'{e}t\'{e} Mathematique Fran\c{c}ais,
  \textbf{129} (2001), 115--139.

\bibitem{Yu:200ve}
\emph{G.~Yu}, The coarse {B}aum-{C}onnes conjecture for spaces which admit a
  uniform embedding into {H}ilbert space, Inventiones Mathematicae,
  \textbf{139}(1) (2000), 201--240.

\end{thebibliography}

\vspace{3em}

\noindent
\textsc{Mathematisches Institut, Universit\"at M\"unster, Einsteinstr.\ 62, 48149 M\"unster, Germany}\\
E-mail: \texttt{jan.spakula@uni-muenster.de}\\

\noindent
\textsc{1326 Stevenson Center, Vanderbilt University, Nashville, TN 37240, USA}\\
E-mail: \texttt{rufus.willett@vanderbilt.edu}

\end{document}